\documentclass[a4paper,10pt,oneside]{article}

\usepackage{diagbox}
\usepackage{amssymb}
\usepackage{amsmath,bm}
\usepackage{amsmath,amsthm}
\usepackage[pdftex, pdfpagelabels, bookmarks, hyperindex, hyperfigures,
colorlinks, pagebackref]{hyperref}

\usepackage{multirow}
\usepackage{makecell}
\usepackage{graphicx}
\usepackage{subfigure}
\usepackage{float}
\usepackage{appendix}
\numberwithin{equation}{section}
\newtheorem{theorem}{Theorem}[section]

\newtheorem{definition}[theorem]{Definition}
\newtheorem{lemma}[theorem]{Lemma}
\newtheorem{remark}{Remark}[section]

\begin{document}




\title
{
	\Large\bf Robust  weak Galerkin finite element methods  for linear  elasticity
with continuous  displacement trace approximation
	\thanks{This work was supported by National Natural Science Foundation of China (11771312) and Major Research
	Plan of  National Natural Science Foundation of China (91430105).}
}

\author
{
 Gang Chen\thanks{Email: 569615491@qq.com}\\
	{School of Mathematics Sciences,}\\
  {University of Electronic Science and Technology of China,}\\
 {Chengdu 611731, China}\\
 \\
  \quad Xiaoping Xie \thanks{Corresponding author. Email: xpxie@scu.edu.cn}\\
	{School of Mathematics, Sichuan University, Chengdu 610064, China}
}

\date{}
\maketitle


\begin{abstract}
This paper proposes and analyzes a class of new weak Galerkin (WG) finite element methods for 2- and 3-dimensional  linear  elasticity problems. The  methods use discontinuous piecewise-polynomial approximations of degrees  $k(\geq 0)$  for the stress,  $k+1$ for the displacement, and a continuous piecewise-polynomial approximation of degree $k+1$ for the  displacement trace on the inter-element boundaries, respectively.  After the local elimination of unknowns defined in the interior of elements,  the WG methods result in  SPD systems where the unknowns are only the degrees of freedom describing the continuous trace   approximation. We show that the proposed methods are  robust in the sense that the derived a priori  error estimates are  optimal and uniform with respect to the Lam\'{e} constant $\lambda$. Numerical experiments confirm the theoretical results.

\end{abstract}




\section{Introduction.}
\label{}
Let $\Omega\color{black}\subset\color{black}\mathbb{R}^d \ (d=2,3)$ be a polyhedral region with boundary $\partial\Omega={\Gamma_D\cup\Gamma_N}$, where $meas(\Gamma_D)>0$ and $\Gamma_D\cap\Gamma_N=\emptyset$.
 We consider 
the following   linear isotropic elasticity model:
\begin{eqnarray}
\left\{
   \begin{aligned}
 \mathcal{A}\bm{\sigma}-\bm{\epsilon}(\bm{u})&=\bm{0},\text{ in }\Omega, \label{o1}\\
 \nabla\cdot\bm{\sigma}&=\bm{f},\text{ in }\Omega,  \\
 \bm{u}&=\bm{g}_D,\text{ on }\Gamma_D, \\
  \bm{\sigma}\bm{n}&=\bm{g}_N,\text{ on }\Gamma_N.
  \end{aligned}
\right.
 \end{eqnarray}
Here  $\bm{\sigma}:\Omega\rightarrow  \mathbb{R}^{d\times d}_{sym}$ denotes  the symmetric $d\times d$ stress tensor field,   $\bm{u}:\Omega\to \mathbb{R}^d$  the displacement field,  $\bm{\epsilon}(\bm{u})=\left(\nabla \bm{u}+(\nabla \bm{u})^T\right)/2$ the strain tensor,  and $\mathcal{A}\bm{\sigma} \in \mathbb{R}^{d\times d}_{sym}$ the compliance tensor with
\begin{eqnarray}
\mathcal{A}\bm{\sigma}=\frac{1}{2\mu}\left(\bm{\sigma}-\frac{\lambda}{2\mu+d\lambda}tr(\bm{\sigma})I\right),
\end{eqnarray}
where $\lambda> 0,\mu>0$ are the Lam\'e coefficients,  $tr(\bm{\sigma})$ denotes the trace of $\bm{\sigma}$, and $I$ is the $d\times d$ identity matrix.
$\bm{f}$ is the  body force  acting on $\Omega$, $\bm{n}$  is   the unit outward
vector normal to $\Gamma_N$, and
  $\bm{g}_D$ and $\bm{g}_N$   are   the surface displacement on $\Gamma_D$ and the surface traction on $\Gamma_N$, respectively.


It is well-known \cite{Vogelius1983} that conforming finite element methods of displacement types suffer from a performance deterioration, called  Poisson-locking,  as  the material becomes incompressible or, equivalently, the Lam\'e constant $\lambda$ tends to $\infty$.

When using  a mixed finite element method based on  Hellinger-Reissner variational principle to solve the model  (\ref{o1}), the combination of stress  and displacement approximation is required to satisfy  two stability conditions, i.e.  a coercivity condition and an inf-sup condition; see, e.g. \cite{Arnold;Awanou2005, Arnold;Brezzi;Jr1984, Arnold;Falk1988, Arnold;Jr;1984, Arnold.D;Winther.R2002, Arnold.D;Winther.R2003, book1,Brezzi.F;Fortin.M1991, Chen.S2011, Fraejis1965,Hu;Shi2007, Johnson;1978,Stenberg1986,Stenberg1988, Xie.X;Xu.J2011}. These  stability constraints usually lead to complicated construction of finite element combinations, in which the stress finite elements are of many degrees of freedom.
In particular,   when dealing with nearly incompressible materials, one needs 
some more-severe stability condition  for the finite element combination  so as to derive  a uniformly stable mixed method which is free from Poisson-locking.

Due to the relaxation of function continuity,  hybrid stress/strain finite element methods based on generalized variational principles \cite{Pian1964, pian1984rational,Punch-Atluri1984, Pian1995, Reddy-Simo1995, Simo-Rifai1990, xie2004optimization, Zhou-Nie, Zhou-Xie2002} are a class of  special mixed approaches that allow   for piecewise-independent approximation to the  stress solution, and thus   lead to   symmetric and positive definite (SPD) discrete systems of nodal displacements after  the local elimination of the stress unknowns defined in the  interior of the elements.
We refer to  \cite{ Bai2016,Braess1998, Braess-C-Reddy2004,yu2011uniform, Li-Xie-Zhang2017}  for the  stability and error estimation of some robust 4-node hybrid  stress/strain  quadrilateral/rectangular elements    that hold uniformly with respect to the the Lam\'e constant $\lambda$.

 The hybridizable discontinuous Galerkin (HDG) framework, proposed in\cite{Cockburn;unified HDG;2009}  for second order elliptic problems, provides a unifying strategy for hybridization of finite element methods. In the HDG model, the constraint of function continuity on the inter- element boundaries is relaxed by introducing Lagrange multipliers defined on the the inter-element boundaries. Similar to the  hybrid stress/strain finite element methods,  by the local elimination of the unknowns defined in the interior of elements,  the HDG  method  results in   a SPD system where the unknowns are only the globally coupled degrees of freedom describing the introduced Lagrange multipliers. We refer to \cite{projection_based_hdg, cond_super_con_hdg,  Li;hdg;2014}
 for the analysis of several HDG methods   for diffusion equations.

The first HDG method for linear
elasticity was proposed in \cite{Soon2008,Soon;2009} and analyzed in \cite{HDG1}, where strongly symmetric stresses and piecewise-polynomial
approximations of degree $k$ in all variables are used.  The method converges optimally for the  displacement and  the antisymmetric part of the displacement gradient, but sub-optimally for the    stress and the strain, i.e.  the symmetric part of the displacement gradient, with $1/2$- order loss of accuracy.
In  \cite{HDG2} an   HDG method was presented based on a strong symmetric stress formulation, where the piecewise-polynomial approximations of degrees  $k(\geq 1)$,  $k+1$ and $k$ are used for the stress, displacement and the  numerical trace of the displacement, respectively. The method was  shown   to yield optimal convergence for both the displacement and stress approximations with the constants in the upper bounds of the errors depending on $\lambda$; In particular, a suboptimal convergence rate for the stress approximation was shown to hold uniformly   with respective to  $\lambda$.  We refer to \cite{Cockburn;Shi2013, Kabaria;Lew;Cockburn?, Nguyen;Peraire;Cockburn2011,Nguyen;Peraire2012} for several related HDG methods for linear elasticity, nonlinear elasticity and elasto-dynamics. We also
  refer to  \cite{LDG} for a local discontinuous Galerkin (LDG) method with strongly symmetric stresses  for linear elasticity.

Closely related to the HDG framework is the weak Galerkin (WG) method developed in \cite{WG1,WG2} for second-order elliptic problems. The WG method   is designed by using a weakly defined gradient operator over functions with discontinuity, and allows the use of totally discontinuous functions in the finite element procedure.   Similar to the hybrid stress/strain  and HDG methods,  the WG scheme leads to a SPD system through the local elimination of unknowns defined in the interior of elements.
We refer to \cite{WG3, WG4, WG5,MY1,MY2,WG-my-forth, Wang;;Zhang2016} for   applications of the WG method to some other partial  differential equations,  and refer to \cite{Chen-Wang-Wang-Ye2014, Li-Xie2014WG,Li;Xie;2015;HDG, Li;bpx;2014} for fast solvers of WG methods.

%

In a very recent work \cite{elasticity}, 
we developed a  class of WG methods  with strong symmetric stresses  for the model \eqref{o1} on  polygon or polyhedral meshes, where discontinuous
piecewise-polynomial approximations of degrees $k(\geq1)$,  $k+1$ are used for the stress,  the displacement, respectively, and discontinuous piecewise-polynomial approximation  of degree $k$ is used for  the displacement trace on the inter-element boundaries.
Optimal convergence rates for both the displacement and   stress approximations  are obtained which hold uniformly with respective to the Lam\'{e} constant $\lambda$. We note that the methods in  \cite{elasticity}, as well as that in \cite{HDG2},  are not applicable to the   lowest order case  $k=0$.

In this paper, we  shall  propose a class of new weak Galerkin  method  with strong symmetric stresses for the model  (\ref{o1})  on  polygon or polyhedral meshes. We use
 discontinuous piecewise-polynomial approximations of degrees $k(\geq0)$  for the stress,  $k+1$ for the displacement, and a continuous piecewise-polynomial approximation of degree $k+1$ for the  displacement trace on the inter-element boundaries, respectively.  Compared with the WG methods in \cite{elasticity},  the new methods have the following features.
\begin{itemize}
\item  They adopts continuous approximation for the displacement trace, while the methods in   \cite{elasticity} use discontinuous trace approximation.

\item They allow the   lowest order case  $k=0$.

\item  They yield optimal convergence rates for both the displacement and   stress approximations which are uniformly with respective to the Lam\'{e} constant $\lambda$.

\item  After the local elimination of unknowns
defined in the interior of elements,  the unknowns  of the resultant  SPD systems  of the new WG method are only the degrees of freedom describing the continuous piecewise-polynomial approximation of the displacement trace on the inter-element boundaries. Especially, the SPD systems is are of smaller sizes than the corresponding systems of the methods  in \cite{elasticity}.

\end{itemize}



The rest of this paper is arranged   as follows. In Section 2 we introduce notation and      WG  finite element schemes. Section 3 derives stability results of the methods. Sections 4 and 5 are  devoted to the a priori  error estimation for the displacement and   stress approximations, respectively. Finally, Section 6 provides   numerical results.

%
%
%
%
%
%
%
%
%
%
%
%
%
%
%
%
%

Throughout this paper,  we use $a\lesssim b$
 to denote $a\le Cb$, where   $C$ is a positive constant independent of the mesh parameters $h$, $h_T$, $h_E$ and the Lam\'e coefficients $\lambda$ and $\mu$.

\section{WG finite element scheme}
\subsection{Notations and preliminary results}
\label{}
%
%
%

For any bounded domain $\Lambda\color{black}\subset\color{black} \mathbb{R}^s$ $(s=d,d-1)$, let $H^{m}(\Lambda)$ and $H^m_0(\Lambda)$  denote the usual  $m^{th}$-order Sobolev spaces on $\Lambda$, and $\|\cdot\|_{m, \Lambda}$, $|\cdot|_{m,\Lambda}$  denote the norm and semi-norm on these spaces.
We use $(\cdot,\cdot)_{m,\Lambda}$ to denote the inner product of $H^m(\Lambda)$, with $(\cdot,\cdot)_{\Lambda}:=(\cdot,\cdot)_{0,\Lambda}$.
When $\Lambda=\Omega$, we denote $\|\cdot\|_{m }:=\|\cdot\|_{m, \Omega}, |\cdot|_{m}:=|\cdot|_{m,\Omega}$, $(\cdot,\cdot):=(\cdot,\cdot)_{\Omega}$.  In particular, when $\Lambda\in \mathbb{R}^{d-1}$, we use $\langle\cdot,\cdot\rangle_{\Lambda}$ to replace $(\cdot,\cdot)_{\Lambda}$.
We note that bold face fonts will be used for vector (or tensor) analogues of the Sobolev spaces along with vector-valued (or tensor-valued) functions. For an integer $k\ge 0$, $\mathbb{P}_k(\Lambda)$  denotes the set of all \color{black} polynomials \color{black} defined on $\Lambda$ with degree not greater than $k$.   In addition, we introduce the following two spaces:
$$L^2_0(\Lambda):=\{v| v\in L^2(\Lambda):= H^0(\Lambda), (v,1)_{\Lambda}=0\},$$
$$\bm{H}(div,\Lambda):=\{\bm{v}| \bm{v}\in [L^2(\Lambda)]^d,\nabla\cdot\bm{v}\in L^2(\Lambda)\}.$$

Let $\mathcal{T}_h=\bigcup\{T\}$  be a shape regular partition (to be defined later) of the domain $\Omega$ consists of arbitrary polygons, \color{black} such that each (open) boundary edge (or face) belongs either to $\Gamma_D$, or to $\Gamma_N$, and there should be at least two edges belong the interior of $\Gamma_S$ $(S=D,N)$ if $\Gamma_S\neq\emptyset$. \color{black}
We note that $\mathcal{T}_h$ can be a conforming partition or a nonconforming partition which allows hanging nodes.

For any $T\in\mathcal{T}_h$, we let $h_T$ be the infimum of the diameters of circles (or spheres) containing $T$ and denote the mesh size $h:=\max_{T\in\mathcal{T}_h}h_T$.
\color{black} 
An edge (or face) $E$ on the boundary  $\partial T$ of $T$ is called a proper edge (or face) if the endpoints (or vertexes) of the edge (or face) $E$ are the nodes of $\mathcal{T}_h$ and no other nodes of $\mathcal{T}_h$ are on $E$. See Figure \ref{fig0} for example, $EF$, $FH$ and $HI$ are proper edges, while $EH$, $FI$ and $EI$ are not. \color{black}
Let $\mathcal{E}_h=\bigcup\{E\}$ be the union of all  proper  edges (faces) of $T\in\mathcal{T}_h$.  We denote by $h_E$ the length of  edge $E$ if $d=2$ and the  infimum of the diameters of circles  containing face $E$ if $d=3$.
For all   $T\in\mathcal{T}_h$ and $E\in\mathcal{E}_h$, we denote by $ \bm{n}_T $ and $\bm{n}_E$   the  unit outward  normal vectors along $\partial T$ and  $E$, respectively.

\begin{figure}[!h]
\centering
\includegraphics[height=5cm ,width=7cm,angle=0]{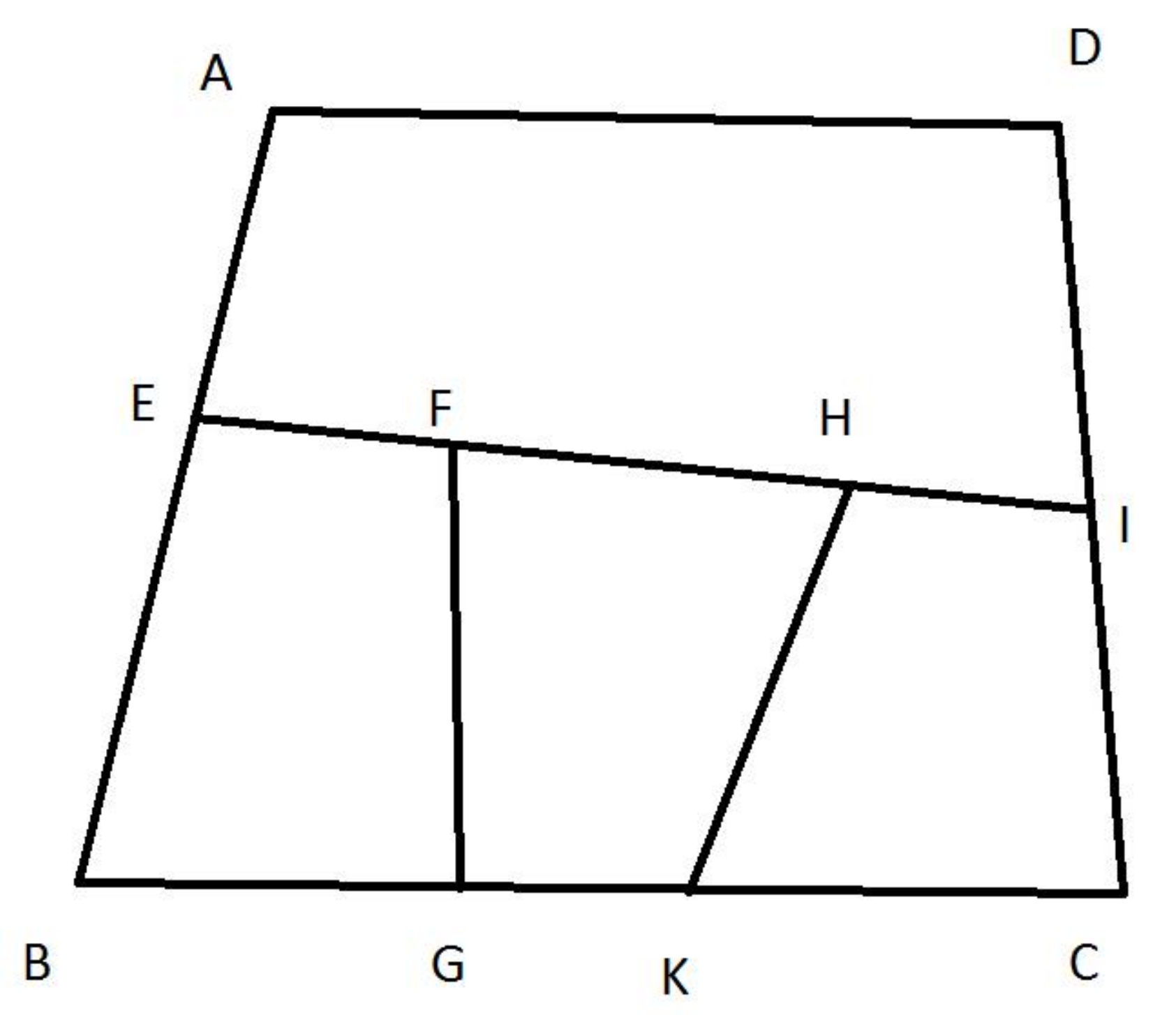}
\caption{\label{fig0} Demonstrating of proper edges in 2D }
\end{figure}

The partition $\mathcal{T}_h$ is called shape regular in the sense that   assumptions \textbf{M1}-\textbf{M2} hold true.
\begin{itemize}

\item \textbf{M1} (Star-shaped elements). There exists a positive constant $\theta_*$ such that the following holds: for each element $T\in\mathcal{T}_h$, there exists a point $M_T\in T$ such that $T$ is star-shaped with respect to every point in the circle (or  sphere) of center $M_T$ and radius $\theta_* h_T$.

\item \textbf{M2} (Edges or faces). There exist a positive constant $l_*$ such that: every element $T\in\mathcal{T}_h$, the distance between any two vertexes (include the hang nodes) is $\ge l_* h_T$.
%
%
%

\end{itemize}


%
When $d=2$,
 for every $T\in\mathcal{T}_h$,
 we connect $M_T$ and $T$'s vertexes (including hang nodes) to get a set of triangles, $w(T)$;
 when $d=3$, for every $T, T^{'}\in\mathcal{T}_h$ and every face $E\subset\partial T\bigcap \partial T^{'}$, we choose any vertex $A$ on $E$ and   connect $A$ to  the rest of $E$'s vertexes (including hang nodes) to get a set of
 triangles, $v(E)$, then we connect $M_T$ and every $v(E)$ ($E\subset\partial T$) to get a set of tetrahedrons, $w(T)$.
Set
 \begin{equation}\label{T^*_h}
 \mathcal{T}^*_h:=\bigcup_{T\in\mathcal{T}_h}w(T).
 \end{equation}
  We note that $\mathcal{T}^*_h$ is shape regular due to \textbf{M1} and \textbf{M2}.
Let $\mathcal{E}^{*}_h$ be the union of all the edges (faces) of $\mathcal{T}_h^*$.
 Notice that when $d=2$, it holds $\mathcal{E}^{*}_h=\mathcal{E}_h$.

For any nonnegative integer $j$, define
$$ \mathbb{P}_{j}(\mathcal{T}_h):=\left\{ q_h: \ q_h|_T\in \mathbb{P}_{j}(T), \ \forall T\in  \mathcal{T}_h\right\}.$$
The space $\mathbb{P}_{j}(\mathcal{T}_h^*)$ is defined similarly.

By using the inverse inequality and trace inequality on the shape regular simplex mesh $\mathcal{T}^*_h$, it is easy to get the following inverse inequality and trace inequality on shape regular polygon mesh $\mathcal{T}_h$.

\begin{lemma} \label{lemma21}
For all $T\in\mathcal{T}_h$ and any given nonnegative integer $j$, the following inequalities hold true: 
\begin{eqnarray}
|q_h|_{1,T} & \lesssim&  h_T^{-1}\|q_h\|_{0,T}, \quad \forall q_h\in \mathbb{P}_j(T),\label{inverse}\\
\|q\|_{0,\partial T}&\lesssim& h_T^{-1/2}\|q\|_{0,T}+h_T^{1/2}|{q}|_{1,{T}},\quad \forall q\in H^1(T). \label{partial1}
\end{eqnarray}
\end{lemma}

For  every $T\in\mathcal{T}_h$, let  $\gamma_T:=\frac{h_T}{\varrho_{T,max}}$ be the chunkiness parameter of $T$, where $\varrho_{T,max}$ denotes the  supremum
of the radius of a sphere  with respect to which $T$ is star-shaped. Then, in view of \textbf{M1},
we have  $2\leq \gamma_T \le \frac{h_T}{\theta_* h_T}=\theta_*^{-1}$, i.e. $\gamma_T$ is independent of $h_T$. Thus, by  (Lemma 4.3.8, \cite{book0}) we obtain the following conclusion.

\begin{lemma} \label{lemma22}For all $T\in\mathcal{T}_h$ and $v\in H^{m}(T)$ with $m\ge1$, there exists $I_{m-1}v\in P_{m-1}(T)$ 
such that
\begin{eqnarray}
|v-I_{m-1}v|_{s,T}\lesssim h_T^{m-s}|v|_{m,T}, \quad \text{ for } 0\leq s\leq m.\label{est2.3} 
\end{eqnarray}
\end{lemma}

In Appendix \ref{A1} we construct a  modified Scott-Zhang interpolation operator, ${\mathcal{I}}_{k+1}$, and derive some properties.

\begin{lemma}\label{lemma23}  For any integer $k\geq0$, there exists an interpolation operator ${\mathcal{I}}_{k+1}: H^1(\Omega)\to H^1(\Omega)\cap \mathbb{P}_{k+1}(\mathcal{T}_h^*)$,  such that 
the following properties hold:  
\begin{eqnarray}
&&\langle{\mathcal{I}}_{k+1}{v},1\rangle_{\Gamma}=\langle{v},1 \rangle_{\Gamma}  \text{ with } \Gamma=\Gamma_D, \Gamma_N,\quad \forall {v}\in H^1(\Omega), \label{2.5}\\
&&{\mathcal{I}}_{k+1}{v}|_{\Gamma_D}={0}, \quad \forall {v}\in H^1_D(\Omega):=\left\{v\in H^1(\Omega):\ v|_{\Gamma_D}=0 \right\},\label{2.6}\\
&&|{v}-{\mathcal{I}}_{k+1}v|_{s,T}\lesssim h_T^{m-s}|v|_{m,S(T)},\quad \forall {v}\in H^{m}(\Omega), \ T\in\mathcal{T}_h^*,\label{2.7}\\
&&\|{v}-{\mathcal{I}}_{k+1}{v}\|_{0,\partial T}\lesssim h_T^{m-1/2}|{v}|_{m,S(T)},\quad \forall {v}\in H^{m}(\Omega), \  T\in\mathcal{T}_h^*,\label{2.8}
\end{eqnarray}
where  $s $, $ m$ are integers satisfying $1\leq m $ and $0\le s\le m\le k+2,$ and 
\begin{eqnarray}
S(T):=\left\{
\begin{array}{ll}
S^{'}(T), &\text{ if } k+1\ge d,\\
\bigcup\limits_{\overline{T^{'}}\bigcap
{S^{'}(T)}\neq\emptyset, T^{'}\in\mathcal{T}_h^* }\overline{T^{'}}, &\text{ if } k+1< d
\end{array}
\right.
\end{eqnarray}
with
\begin{eqnarray}
S^{'}(T):=\bigcup\limits_
{ \overline{T^{'}}\bigcap
\overline{T}\neq\emptyset, T^{'}\in\mathcal{T}_h^*} \overline{T^{'}}.
\end{eqnarray}
\end{lemma}
\begin{proof} The proofs of \eqref{2.5}-\eqref{2.7} are given in Appendix \ref{A1}, and the estimate \eqref{2.8} follows from \eqref{2.7} and Lemma \ref{lemma21}.
\end{proof}

\color{black}

\subsection{Discrete weak gradient/divergence operators }

We follow \cite{WG4} to introduce the definitions of the discrete weak gradient/divergence operators.  For $T\in\mathcal{T}_h$, 
denote by $\mathcal{V}(T)$ a space of weak functions on $T$ with
\begin{eqnarray}
\mathcal{V}(T)=\{v=\{v_i,v_b\}:v_i\in L^2(T),v_b\in H^{1/2}(\partial T)\}.
\end{eqnarray}

Let $\bm{G}(T) \subset  \bm{H}(div,T)$ be a local finite dimensional vector space. We define the discrete weak gradient operator $\nabla_{w,G,T}: \mathcal{V}(T)\rightarrow \bm{G}(T)$  as follows.
\begin{definition} For all $v\in \mathcal{V}(T)$,  its  discrete weak gradient   $\nabla_{w,G,T}v\in \bm{G}(T)$ satisfies the equation
\begin{eqnarray}
(\nabla_{w,G,T}v,\bm{\tau})_T=-(v_i,\nabla\cdot \bm{\tau})_T+\langle v_b,\bm{\tau}\cdot  \bm{n}_T\rangle_{\partial T}, \quad\forall \bm{\tau}\in \bm{G}(T).
\end{eqnarray}
\end{definition}
Then we define the global  discrete weak gradient operator $\nabla_{w,G}$ with $$\nabla_{w,G}|_T=\nabla_{w,G,T},\forall T\in\mathcal{T}_h.$$
 In particular, if $\bm{G}(T)=[\mathbb{P}_r(T)]^d$, we write $\nabla_{w,r}:=\nabla_{w,G}$. For a  vector   $\bm{v}=(v_1,\cdots,v_d)^T\in [\mathcal{V}(T)]^d$,  we define the weak gradient $\nabla_{w,r}\bm{v} $  as
 $$\nabla_{w,r}\bm{v}=(\nabla_{w,r}v_1,\dotsm\nabla_{w,r}v_d)^T.$$

%

Let $\mathcal{W}(T)$ be a space of weak vector-valued functions on  $T$ with
\begin{eqnarray}
\mathcal{W}(T):=\{\bm{v}=\{\bm{v}_i,\bm{v}_b\}:\bm{v}_i\in [L^2(T)]^d,\bm{v}_b\cdot\bm{n}_T \in H^{-1/2}(\partial T)\},
\end{eqnarray}
and 

   ${G}(T)\subset H^1(T)$ be a local finite dimensional space.  We define the discrete weak divergence operator $\nabla_{w,G,T}\cdot: \mathcal{W}(T)\rightarrow  {G}(T)$  as follows.
\begin{definition} For any $\bm{v}\in \mathcal{W}(T)$,  its  discrete weak divergence   $\nabla_{w,G,T}\cdot\bm{v}\in {G}(T)$ satisfies the equation
\begin{eqnarray}
(\nabla_{w,G,T}\cdot \bm{v},\tau)_T=-(\bm{v}_i,\nabla{\tau})_T+\langle \bm{v}_b\cdot\bm{n}_T ,{\tau}\rangle_{\partial T}\quad \forall \tau\in G(T).
\end{eqnarray}
\end{definition}
Then we define  the global  discrete weak divergence operator  $\nabla_{w,G}\cdot$ with $\nabla_{w,G}\cdot|_T=\nabla_{w,G,T}\cdot,\forall T\in\mathcal{T}_h$. In particular, if ${G}(T)=\mathbb{P}_r(T)$, we write $\nabla_{w,r}\cdot:=\nabla_{w,G}\cdot$.
\color{black}

\subsection{New WG finite element schemes}

For any $T\in\mathcal{T}_h$, $ E\in\mathcal{E}_h$ and any   nonnegative integer $j$, let $Q_{j}^i: L^2(T)\rightarrow \mathbb{P}_{j}(T)$ and $Q_{j}^b: L^2(E)\rightarrow \mathbb{P}_j(E)$ be the usual $L^2$ projection operators. Vector or tensor  analogues of $Q_{j}^i$ and $Q_{j}^b$ are denoted by  $\bm Q_{j}^i$ and $\bm Q_{j}^b$, respectively.

For any   integer $k\ge 0$, we introduce  the following  finite dimensional spaces:
\begin{eqnarray}
  \bm{U}_{hi}&:=&\{\bm{v}_{hi}\in [L^2(\Omega)]^d:\ \bm{v}_{hi}|_T\in [\mathbb{P}_{k+1}(T)]^{d},\ \forall T\in\mathcal{T}_h\},\\
  \bm{U}_{hb}&:=&\{\bm{v}_{hb}\in [C^0(\mathcal{E}^{*}_h)]^d:\  \bm{v}_{hb}|_E\in [\mathbb{P}_{k+1}(E)]^d,\ \forall E\in\mathcal{E}^{*}_h\},    \\
 \bm{U}^{\tilde{\bm{g}}}_{hb}&:=&\{\bm{v}_{hb}\in \bm{U}_{hb}:\ \bm{v}_{hb}|_{\Gamma_D}= \bm{\mathcal{I}}_{k+1}\tilde{\bm{g}}\}    \text{ with }\tilde{\bm{g}}=\bm{g}_D,\bm{0},\\
 \bm{\Sigma}_{hi}&:=&\{\bm{\tau}_{hi}\in [L^2(\Omega)]^{d\times d}:  \bm{\tau}_{hi}^T= \bm{\tau}_{hi} \text{  and } \bm{\tau}_{hi}|_T\in [\mathbb{P}_k(T)]^{d\times d},\forall T\in\mathcal{T}_h\}.\nonumber\\
 \end{eqnarray}
 For simplicity we set
\begin{equation}
 \bm{U}_{h}:=\bm{U}_{hi}\times \bm{U}_{hb},\quad   \bm{U}_{h}^{\tilde{\bm{g}}}:= \bm{U}_{hi}\times \bm{U}_{hb}^{\tilde{\bm{g}}}    \text{ with }\tilde{\bm{g}}=\bm{g}_D,\bm{0}.
 \end{equation}
In  the case of $k+1<d$,  the following space is also required  for stabilization:
 \begin{eqnarray}
 {\Sigma}^{tr}_{hb}&:=&\{{\tau}^{tr}_{hb}\in C^0(\mathcal{E}^{*}_h/\partial\Omega): {\tau}^{tr}_{hb}|_E\in \mathbb{P}_{1}(E), \forall E\in\mathcal{E}^{*}_h/\partial\Omega\}.
\end{eqnarray}
%
Then the stabilized WG (new WG) finite element discretization of the elasticity model \eqref{o1} is given in the following two cases:

(i) Case of $k+1<d$:

 Find  $\bm{\sigma}_{hi}\in \bm{\Sigma}_{hi}, \sigma_{hb}^{tr}\in \Sigma_{hb}^{tr}, \bm{u}_h:=\{\bm{u}_{hi},\bm{u}_{hb}\}\in  \bm{U}_{h}^{\bm{g}_D}$ such that
\begin{align}
 {a}_h(\bm{\sigma}_{hi},\bm{\tau}_{hi})- {b}_h(\bm{\tau}_{hi},\bm{u}_{h})
 +{z}_h(\bm{\sigma}_{hi},\sigma^{tr}_{hb};\bm{\tau}_{hi},\tau^{tr}_{hb})=0,\quad& \forall \bm{\tau}_{hi}\in \bm{\Sigma}_{hi}, \tau_{hb}^{tr}\in \Sigma_{hb}^{tr}, \label{wg1new}\\
 -{b}_h(\bm{\sigma}_{hi},\bm{v}_{h})-{s}_h(\bm{u}_h,\bm{v}_h)=(\bm{f},\bm{v}_{hi})-\langle \bm{g}_N,\bm{v}_{hb}\rangle_{\Gamma_N},\quad &\forall \bm{v}_h:=\{\bm{v}_{hi},\bm{v}_{hb}\}\in  \bm{U}_{h}^{\bm{0}};\label{wg2new}
 \end{align}

(ii) Case of $k+1\ge d$:

Find  $\bm{\sigma}_{hi}\in \bm{\Sigma}_{hi},  \bm{u}_h:=\{\bm{u}_{hi},\bm{u}_{hb}\}\in  \bm{U}_{h}^{ \bm{g}_D}$ such that
\begin{align}
 {a}_h(\bm{\sigma}_{hi},\bm{\tau}_{hi})- {b}_h(\bm{\tau}_{hi},\bm{u}_{h})
 =0, \quad& \forall \bm{\tau}_{hi}\in \bm{\Sigma}_{hi}, \label{wg1}\\
 -{b}_h(\bm{\sigma}_{hi},\bm{v}_{h})-{s}_h(\bm{u}_h,\bm{v}_h)=(\bm{f},\bm{v}_{hi})-\langle \bm{g}_N,\bm{v}_{hb}\rangle_{\Gamma_N},   &\forall \bm{v}_h\in  \bm{U}_{h}^{ \bm{0}}.\label{wg2}
 \end{align}
Here
\begin{eqnarray}
 {a}_h(\bm{\sigma}_{hi},\bm{\tau}_{hi})&:=&(\mathcal{A}\bm{\sigma}_{hi},\bm{\tau}_{hi}),\\
 {b}_h(\bm{\tau}_{hi},\bm{v}_{h})&:=&(\bm{\tau}_{hi},\nabla_{w,k}\bm{v}_{h}),\\
 {s}_h(\bm{u}_h,\bm{v}_h)&:=&\langle \alpha          (\bm{u}_{hi}-\bm{u}_{hb}),\bm{v}_{hi}-\bm{v}_{hb}\rangle_{\partial\mathcal{T}_h},\label{227}\\
  {z}_h(\bm{\sigma}_{hi},\sigma^{tr}_{hb};\bm{\tau}_{hi},\tau^{tr}_{hb})&:=&
 \langle\beta (tr(\bm{\sigma}_{hi})-\sigma^{tr}_{hb}),
                                                         tr(\bm{\tau}_{hi})- \tau^{tr}_{hb}
  \rangle_{\partial\mathcal{T}_h/\partial\Omega},
\end{eqnarray}
and the parameter $\alpha$ and $\beta$ are taken as
\begin{eqnarray}
\alpha|_E&:=&2\mu h_E^{-1} \quad \text{for any } E\in\mathcal{E}_h ,\\
\beta|_E&:=&(2\mu)^{-1}h_E \quad \text{for any } E\in\mathcal{E}_h/\partial\Omega.
 \end{eqnarray}




\begin{remark}
We note that the stabilization term $z_h(\cdot,\cdot;\cdot,\cdot)$ is the key to   locking-free performance when $k+1<d$. 
\end{remark}
\begin{remark}\label{rem2.2}
The WG methods in   \cite{elasticity} are based on the following formulations with   $k\geq 1$. Find  $\bm{\sigma}_{hi}\in \bm{\Sigma}_{hi},  \bm{u}_h=\{\bm{u}_{hi},\bm{u}_{hb}\}\in \tilde{\bm{U}}_{h}^{ \bm{g}_D} := \bm{U}_{hi}\times \tilde{\bm{U}}_{hb}^{{\bm{g}}_D} $ such that
\begin{align}
 {a}_h(\bm{\sigma}_{hi},\bm{\tau}_{hi})- {b}_h(\bm{\tau}_{hi},\bm{u}_{h})
 =0, \quad& \forall \bm{\tau}_{hi}\in \bm{\Sigma}_{hi}, \label{wg1*}\\
 -{b}_h(\bm{\sigma}_{hi},\bm{v}_{h})-\tilde{s}_h(\bm{u}_h,\bm{v}_h)=(\bm{f},\bm{v}_{hi})-\langle \bm{g}_N,\bm{v}_{hb}\rangle_{\Gamma_N},   &\forall \bm{v}_h\in   \tilde{\bm{U}}_{h}^{ \bm{0}} .\label{wg2*}
 \end{align}
where
$$ \tilde{\bm{U}}_{hb}^{\tilde{\bm{g}}}:=  \{\bm{v}_{hb}\in [L^2(\mathcal{E}_h)]^d:\  \bm{v}_{hb}|_{\Gamma_D}= \bm Q^b_{k}\tilde{\bm{g}}, \ \bm{v}_{hb}|_E\in [\mathbb{P}_{k}(E)]^d , \forall  E\in\mathcal{E}_h \}
$$
 with $\tilde{\bm{g}}=\bm{g}_D,\bm{0}$, and
 $$  \tilde{s}_h(\bm{u}_h,\bm{v}_h):=\langle \alpha          (\bm Q^b_{k}\bm{u}_{hi}-\bm{u}_{hb}),\bm Q^b_{k}\bm{v}_{hi}-\bm{v}_{hb}\rangle_{\partial\mathcal{T}_h}.$$
 In fact, by following the routine of analysis  in \cite{elasticity},  the case of $k=0$ for \eqref{wg1*}-\eqref{wg2*}  also works if $ \tilde{\bm{U}}_{hb}^{\tilde{\bm{g}}}$ in this case  is modified as
 \begin{eqnarray*}
 \tilde{\bm{U}}_{hb}^{\tilde{\bm{g}}}:=  \{\bm{v}_{hb}\in [L^2(\mathcal{E}_h)]^d:\  \bm{v}_{hb}|_{\Gamma_D}= \bm Q^b_{1}\tilde{\bm{g}}, \ \bm{v}_{hb}|_E\in [\mathbb{P}_{1}(E)]^d , \forall  E\in\mathcal{E}_h \}.
    \end{eqnarray*}
    As a result,  for  the WG scheme  \eqref{wg1*}-\eqref{wg2*} with $k\geq 0$,  the following error estimates hold true (\cite{elasticity}):
\begin{eqnarray*}
\|\bm{\sigma}-\bm{\sigma}_{hi}\|_0&\lesssim&  h^{k+1}(|\bm{\sigma}|_{k+1}+\mu|\bm{u}|_{k+2}), \\
\|\nabla\bm{u}-\nabla_h\bm{u}_{hi}\|_0&\lesssim& h^{k+1}(\mu^{-1}|\bm{\sigma}|_{k+1}+|\bm{u}|_{k+2} )\label{in47g}.
\end{eqnarray*}
 In addition,    under  the regularity assumption (\ref{rg}), it holds
\begin{eqnarray*}
\|\bm{u}-\bm{u}_{hi}\|_0\lesssim h^{k+2}(\mu^{-1}|\bm{\sigma}|_{k+1}+|\bm{u}|_{k+2} ).
 \end{eqnarray*}

\color{black}

\end{remark}

To establish error estimates for the proposed WG schemes, we need the following approximation and stability results  for the $L^2$-projections $Q^i_j$ and $Q^b_j$ with nonnegative integer $j$.


\color{black}
\begin{lemma} \cite{elasticity} \label{lemma2.6}  Let $m$ be an integer with $1\le m\le j+1$. For all $T\in\mathcal{T}_h$, $E\in\mathcal{E}_h$, it holds
\begin{eqnarray}
\small
\|Q_{j}^iv\|_{0,T}&\le&\|v\|_{0,T},\forall v\in L^{2}(T),\label{pr2}\\
\|Q_{j}^bv\|_{0,E}&\le&\|v\|_{0,E},\forall v\in L^{2}(E),\label{pr3}\\
\|v-Q_{j}^bv\|_{0,\partial T}&\lesssim& h_T^{m-1/2}|v|_{m,T},\forall v\in H^{m}(T),\label{pr111}\\
|v-Q_{j}^iv|_{s,T}&\lesssim& h_T^{m-s}|v|_{m,T},\forall v\in H^{m}(T),\ 0\le s \le m,\label{pr1}\\
\|\nabla^s(v-Q_{j}^iv)\|_{0,\partial T}&\lesssim& h_T^{m-s-1/2}|v|_{m,T},\forall v\in H^{m}(T), \color{red}1\color{black}\le s+1 \le m. \label{pr11}
\end{eqnarray}

\end{lemma}
%
%

\section{Stability}

We first show the following inf-sup stability condition holds for the bilinear form ${b}_h(\cdot,\cdot)$.
\begin{theorem} \label{th31} Denote $\bm{\epsilon}_{w,k}(\bm{v}_{h}):=\left(\nabla_{w,k} \bm{v}_{h}+(\nabla_{w,k} \bm{v}_{h})^T\right)/2$. Then, for all $\bm{v}_{h}\in\bm{U}_{h}$,  it holds
\begin{eqnarray}
\sup_{\bm{\tau}_{hi}\in \bm{\Sigma}_{hi}}\frac{{b}_h(\bm{\tau}_{hi},\bm{v}_{h})}{\|\bm{\tau}_{hi}\|_0}\ge \|\bm{\epsilon}_{w,k}(\bm{v}_{h})\|_0.
\end{eqnarray}
\end{theorem}

\begin{proof} The conclusion follows from the fact that $\bm{\epsilon}_{w,k} (\bm{v}_{h})\in \bm{\Sigma}_{hi}$ for any $\bm{v}_{h}\in\bm{U}_{h}$.
\end{proof}

To derive  stability   conditions for the bilinear form $a_h(\cdot,\cdot)$, we need the following lemma.

\begin{lemma} \label{LBB} For all ${p}_{hi}\in \mathbb{P}_k(\mathcal{T}_h)\bigcap L^2(\Omega)$ if $\Gamma_N\not=\emptyset$, and all   $p_{hi}\in \mathbb{P}_k(\mathcal{T}_h)\bigcap L^2_0(\Omega)$ if $\Gamma_N=\emptyset$, it holds
\begin{eqnarray}\label{phi3.2}
\|p_{hi}\|_0^2\lesssim
 \left\{ \begin{array} {ll}
 \sum\limits_{E\in\mathcal{E}_h/\partial\Omega}h_E\|[p_{hi}]\|^2_{0,E}
          +\left[\sup\limits_{\bm{v}_{hi}\in \bm{X}_{hi}}\frac{(\nabla\cdot\bm{v}_{hi},p_{hi})}
             {|\bm{v}_{hi}|_1}\right]^2 & \text{if } k+1<d,\\
\left[\sup\limits_{\bm{v}_{hi}\in \bm{X}_{hi}}\frac{(\nabla\cdot\bm{v}_{hi},p_{hi})}
             {|\bm{v}_{hi}|_1}\right]^2 & \text{if } k+1\geq d,
\end{array}\right.
\end{eqnarray}
where 
\begin{eqnarray*}
\bm{X}_{hi}:= \{\bm{v}_{hi}\in [H^1_D(\Omega)]^d:\bm{v}_{hi}|_E\in [\mathbb{P}_{k+1}(E)]^d,\forall E\in\mathcal{E}_h\}. 
\end{eqnarray*}
\end{lemma}
\begin{proof} \color{black} From Theorem \ref{thmB.3} in Appendix \ref{B1} we know that
  the  Stokes  elements that are stable in $[H^1_0(\Omega)]^d\times L^2_0(\Omega)$  are also stable in  $[H^1_D(\Omega)]^d\times L^2(\Omega)$. 

If $k+1\ge d$, by applying   the stable    Stokes-elements ($\mathbb{P}_{k+1}/\mathbb{P}^{dis}_k$) by  \cite{Scott-Vogelius1985, Qin1994, Zhang.S2005,SV} on the barycenter refined  mesh of \color{black} $\mathcal{T}_h^*$,  the desired  estimate \eqref{phi3.2} follows directly.

If $k+1<d$, we shall make use of the MINI element  $\mathbb{P}_1^b/\mathbb{P}_1$
 on the mesh $\mathcal{T}_h^*$  with $\mathbb{P}_1^b=\mathbb{P}_1\oplus \{bubbles\}$.  Note that in this case we have $k\le 1$. From \cite{adapt}, there exists an   interpolation operator $\mathcal{Q}^*_{1}: \mathbb{P}_{1}(\mathcal{T}^*_h) \rightarrow \mathbb{P}_{1}(\mathcal{T}^*_h)\bigcap H^1(\Omega)$ such that,
for all $q_{hi}\in \mathbb{P}_{1}(\mathcal{T}^*_h)$, 
\begin{eqnarray}
\sum_{T\in\mathcal{T}_h^*}\|q_{hi}-\mathcal{Q}^*_{1}q_{hi}\|^2_{0,T}
\lesssim \sum_{E\in\mathcal{E}^*_h/\partial\Omega}h_E^{1/2}\|[q_{hi}]\|^2_{0,E}.
\end{eqnarray}
  Then, for all ${p}_{hi}\in \mathbb{P}_k(\mathcal{T}_h)\bigcap L^2(\Omega)$ if $\Gamma_N\not=\emptyset$, and all   $p_{hi}\in \mathbb{P}_k(\mathcal{T}_h)\bigcap L^2_0(\Omega)$ if $\Gamma_N=\emptyset$, we have
\begin{eqnarray*}
\|p_{hi}\|_0^2&\lesssim&\|p_{hi}-\mathcal{Q}_{1}^{*} p_{hi}\|^2_0+\|\mathcal{Q}_{1}^{*} p_{hi}\|^2_0\nonumber\\
              &\lesssim&\|p_{hi}-\mathcal{Q}_{1}^{*} p_{hi}\|^2_0
          +\left[\sup_{\bm{v}_{hi}\in [\mathbb{P}^b_1(\mathcal{T}_h^*)\cap H_D^1(\Omega)]^d}\frac{(\nabla\cdot\bm{v}_{hi},\mathcal{Q}_{1}^{*} p_{hi})}
             {|\bm{v}_{hi}|_1}\right]^2\nonumber\\
          &\lesssim&\|p_{hi}-\mathcal{Q}_{1}^{*} p_{hi}\|^2_0
          +\left[\sup_{\bm{v}_{hi}\in [\mathbb{P}^b_1(\mathcal{T}_h^*)\cap H_D^1(\Omega)]^d}\frac{(\nabla\cdot\bm{v}_{hi},p_{hi})}
             {|\bm{v}_{hi}|_1}\right]^2\nonumber\\
             &&+\left[\sup_{\bm{v}_{hi}\in [\mathbb{P}^b_1(\mathcal{T}_h^*)\cap H_D^1(\Omega)]^d}\frac{(\nabla\cdot\bm{v}_{hi},\mathcal{Q}_{1}^{*} p_{hi}-p_{hi})}
             {|\bm{v}_{hi}|_1}\right]^2\nonumber\\
   &\lesssim&\|p_{hi}-\mathcal{Q}_{1}^{*} p_{hi}\|^2_0
          +\left[\sup_{\bm{v}_{hi}\in [\mathbb{P}^b_1(\mathcal{T}_h^*)\cap H_D^1(\Omega)]^d}\frac{(\nabla\cdot\bm{v}_{hi},p_{hi})}
             {|\bm{v}_{hi}|_1}\right]^2\nonumber\\
              &\lesssim&\sum_{E\in\mathcal{E}_h^*/\partial\Omega}h_E \|[p_{hi}]\|^2_{0,E}
          +\left[\sup_{\bm{v}_{hi}\in \bm{X}_{hi}}\frac{(\nabla\cdot\bm{v}_{hi},p_{hi})}
             {|\bm{v}_{hi}|_1}\right]^2\nonumber\\
              &\lesssim&\sum_{E\in\mathcal{E}_h/\partial\Omega}h_E \|[p_{hi}]\|^2_{0,E}
          +\left[\sup_{\bm{v}_{hi}\in \bm{X}_{hi}}\frac{(\nabla\cdot\bm{v}_{hi},p_{hi})}
             {|\bm{v}_{hi}|_1}\right]^2,
\end{eqnarray*}
which completes the proof.
%
\end{proof}


For any $\bm{\tau}\in \mathbb{R}^{d\times d}$, let $\bm{\tau}_D:=\bm{\tau}-\frac{1}{d}tr(\bm{\tau})I$ denote its deviatoric tensor. Then we easily have  the following identities:
for all $\bm{\sigma},\bm{\tau}\in [L^2(\Omega)]^{d\times d}$ and $\bm{v}\in [H^1(\Omega)]^d$,
\begin{eqnarray}
(\mathcal{A}\bm{\sigma},\bm{\tau})&=&\frac{1}{2\mu}(\bm{\sigma}_D,\bm{\tau}_D)+\frac{1}{d(d\lambda+2\mu)}(tr(\bm{\sigma}),tr(\bm{\tau})),\label{D1}\\
\|\bm{\tau}\|^2_0&=&\|\bm{\tau}_D\|_0^2+\frac{1}{d}\|tr(\bm{\tau})\|^2_0,\label{D2}\\
\nabla\cdot\bm{v}I&=&d(\nabla\bm{v}-(\nabla\bm{v})_D),\label{D3}\\
(\bm{\sigma}_D,\bm{\tau})&=&(\bm{\sigma},\bm{\tau}_D).\label{D4}
\end{eqnarray}

In light of the relations \eqref{D1}-\eqref{D4} and Lemma \ref{LBB}, we   can prove the following stability  results.

\begin{theorem} \label{th34} For all $\bm{\sigma}_{hi},\bm{\tau}_{hi}\in \bm{\Sigma}_{hi}$, it holds the continuity condition
\begin{eqnarray}\label{continuity-a}
a_h(\bm{\sigma}_{hi},\bm{\tau}_{hi})\le\frac{1}{2\mu}\|\bm{\sigma}_{hi}\|_0\|\bm{\tau}_{hi}\|_0.
\end{eqnarray}
Moreover, it holds the coercivity condition
\begin{eqnarray}\label{stability-a}
\|\bm{\sigma}_{hi}\|_0^2\lesssim
 \left\{ \begin{array} {l}
\mu a_h(\bm{\sigma}_{hi},\bm{\sigma}_{hi})+\mu\|\alpha^{-1/2}(\bm{\sigma}_{hi}\bm{n}-\bm{\sigma}_{hb}\bm n)\|^2_{0,\partial\mathcal{T}_h} +\mu\|\beta^{1/2} (tr(\bm{\sigma}_{hi})-\sigma^{tr}_{hb}\|^2_{\partial\mathcal{T}_h/\partial\Omega}\\
\hskip7.5cm\  \text{if } k+1<d,\\
\mu a_h(\bm{\sigma}_{hi},\bm{\sigma}_{hi})+\mu\|\alpha^{-1/2}(\bm{\sigma}_{hi}\bm{n}-\bm{\sigma}_{hb}\bm n)\|^2_{0,\partial\mathcal{T}_h}  \qquad \text{if } k+1\geq d,
\end{array}\right.
%
\end{eqnarray}
for all $ \sigma_{hb}^{tr}\in \Sigma_{hb}^{tr}$ and $(\bm{\sigma}_{hi}, \bm{\sigma}_{hb})\in \bm{\Sigma}_{hi}\times [L^2(\partial\mathcal{T}_h)]^{d\times d} $ satisfying
 \begin{eqnarray}
 &&
 \nabla_{w,\color{black}k+1\color{black}}\cdot\{\bm{\sigma}_{hi},\bm{\sigma}_{hb}\}=\bm{0},
 \label{c3.9}\\
&&\langle {\bm\sigma}_{hb}\bm n,\bm{v}_{hb}\rangle_{\partial\mathcal{T}_h}=0,
 \forall \bm{v}_{hb}\in \bm{U}^{D,\bm{0}}_{hb},\label{c3.10}\\
 &&tr(\bm{\sigma}_{hi})\in L^2_0(\Omega) \text{ if }\Gamma_N=\emptyset.\label{c3.11}
  \end{eqnarray}
\end{theorem}
\begin{proof}  The inequality \eqref{continuity-a} follows from the identities (\ref{D1})-(\ref{D2}) and the definition of $a_h(\cdot,\cdot)$.

For $\bm{\sigma}_{hi}\in \bm{\Sigma}_{hi}$ satisfying \eqref{c3.11}, by Lemma \ref{LBB}   we obtain
\begin{eqnarray}\label{313}
\|tr(\bm{\sigma}_{hi})\|^2_0
\lesssim\left\{ \begin{array}{ll}
\sum\limits_{T\in\mathcal{T}_h/\partial\Omega}h_T\|tr(\bm{\sigma}_{hi})-\sigma_{hb}^{tr}\|^2_{0,\partial T}
          +\left[\sup\limits_{\bm{v}_{hi}\in \bm{X}_{hi}}\frac{(\nabla\cdot\bm{v}_{hi},tr(\bm{\sigma}_{hi}))}
             {|\bm{v}_{hi}|_1}\right]^2 & \text{if } k+1<d,\\
\left[\sup\limits_{\bm{v}_{hi}\in \bm{X}_{hi}}\frac{(\nabla\cdot\bm{v}_{hi},tr(\bm{\sigma}_{hi}))}
             {|\bm{v}_{hi}|_1}\right]^2 & \text{if } k+1\geq d,
\end{array}
\right.
\end{eqnarray}
where  in the case of $k+1<d$ we have used   the relation  $[\sigma_{hb}^{tr}]|_E=0$ for any $E\in  \mathcal{E}_h/\partial\Omega$ and the shape-regularity of $\mathcal{T}_h$.

In view of    (\ref{D3}), it holds
\begin{eqnarray*}
(\nabla\cdot\bm{v}_{hi},tr(\bm{\sigma}_{hi}))=(\nabla\cdot\bm{v}_{hi}I_{d\times d},\bm{\sigma}_{hi})=d(\nabla\bm{v}_{hi}-(\nabla\bm{v}_{hi})_D,\bm{\sigma}_{hi}).
\end{eqnarray*}
Since $\bm{v}_{hi}\in \bm{U}^{D,\bm{0}}_{hb}$  for all $\bm{v}_{hi}\in\bm{X}_{hi}$,
by the above relation, Green's formula, the identity (\ref{D4}), \eqref{c3.10}, and the definition of weak divergence, we get
\begin{eqnarray*}
(\nabla\cdot\bm{v}_{hi},tr(\bm{\sigma}_{hi}))&=&-d(\bm{v}_{hi},\nabla_h\cdot\bm{\sigma}_{hi})
                          +d\langle\bm{v}_{hi},\bm{\sigma}_{hi}\bm{n}\rangle_{\partial\mathcal{T}_h}
                          -d(\nabla\bm{v}_{hi},(\bm{\sigma}_{hi})_D)\nonumber\\
                          &=&-d(\bm{Q}^i_{\color{black}k+1\color{black}}\bm{v}_{hi},\nabla_h\cdot\bm{\sigma}_{hi})
                          +d\langle\bm{v}_{hi},\bm{\sigma}_{hi}\bm{n}-\bm{\sigma}_{hb}\bm n\rangle_{\partial\mathcal{T}_h} \nonumber\\
                          &&\quad -d(\nabla\bm{v}_{hi},(\bm{\sigma}_{hi})_D)\nonumber\\
                          &=&-d(\bm{Q}^i_{\color{black}k+1\color{black}}\bm{v}_{hi},\nabla_{w,\color{black}k+1\color{black}}\cdot\{\bm{\sigma}_{hi},\bm{\sigma}_{hb}\})\nonumber\\
                          &&\quad +d\langle\bm{v}_{hi}-\bm{Q}^i_{\color{black}k+1\color{black}}\bm{v}_{hi},\bm{\sigma}_{hi}\bm{n}-\bm{\sigma}_{hb}\bm n \rangle_{\partial\mathcal{T}_h}
                          -d(\nabla\bm{v}_{hi},(\bm{\sigma}_{hi})_D).\nonumber
\end{eqnarray*}
Then, from \eqref{c3.9} and Lemma \ref{lemma2.6} with $m=1$ it follows
\begin{eqnarray}
\frac{(tr(\bm{\sigma}_{hi}),\nabla\cdot\bm{v}_{hi})}{|\bm{v}_{hi}|_1}\lesssim \mu^{1/2}\|\alpha^{-1/2}(\bm{\sigma}_{hi}\bm{n}-\bm{\sigma}_{hb}\bm n)\|_{0,\partial\mathcal{T}_h}
                          +\|(\bm{\sigma}_{hi})_D\|_0,\label{316}
\end{eqnarray}
which, together with  (\ref{313})  and  (\ref{D1}), yields the desired inequality \eqref{stability-a}.
\end{proof}

%


\section{A priori error estimates}\label{s4}

This section  is devoted to the error estimation for the  modified WG schemes \eqref{wg1new}-\eqref{wg2new} and \eqref{wg1}-\eqref{wg2}. 

\begin{lemma}\label{lemma41} For all  $T\in\mathcal{T}_h$, $\bm{\tau}\in  \bm{H}(div,T)$,   $\bm v\in [H^1(\Omega)]^d$, $\bm{\tau}_{hi}\in [\mathbb{P}_k(T)]^d$,  and $\bm{v}_h=\{\bm{v}_{hi},\bm{v}_{hb}\}\in\bm{U}_h$, there holds
\begin{align}
\nabla_{w,\color{black}k+1\color{black}}\cdot\{\bm{Q}^i_{k}\bm{\tau}, \bm{Q}^b_{\color{black}k+1\color{black}}\bm{\tau}\}&=\bm{Q}_{k+1}^i\nabla\cdot\bm{\tau}, \forall \bm{\tau}\in  \bm{H}(div,T) \label{pwg2}\\
(\nabla_{w,k}\{\bm{Q}^i_{k+1}\bm{v}, \bm{\mathcal{I}}_{k+1}\bm{v}\},\bm{\tau}_{hi})_T&=(\nabla\bm{v},\bm{\tau}_{hi})_T+\langle \bm{\mathcal{I}}_{k+1}\bm{v}-\bm{v},\bm{\tau}_{hi}\bm{n}\rangle_{\partial T},\label{pwg1}\\
\small \|\bm{\epsilon}_h(\bm{v}_{hi})\|_0\lesssim
\|\bm{\epsilon}_{w,k}(\bm{v}_{h})\|_0&+\mu^{-1/2}\|\alpha^{1/2}(\bm{Q}^b_k\bm{v}_{hi}-\bm{v}_{hb})\|_{0,\partial\mathcal{T}_h},
 \label{pwg3}
 \end{align}
 where
$\bm{\epsilon}_{w,k}(\bm{v}_{h}):=\left(\nabla_{w,k} \bm{v}_{h}+(\nabla_{w,k} \bm{v}_{h})^T\right)/2$, and $\bm{\epsilon}_h(\bm{v}_{hi}):=\left(\nabla_h \bm{v}_{hi}+(\nabla_h \bm{v}_{hi})^T\right)/2$ with $\nabla_h$ denoting the broken gradient operator with respect to $\mathcal{T}_h$.

\end{lemma}
\begin{proof}
For any $T\in\mathcal{T}_h$, $\bm{v}_{hi}\in [\mathbb{P}_{k+1}(T)]^d$ and $\bm{\tau}\in \bm{H}(div,T)$, we use the definition of weak divergence, the orthogonality of $\bm{Q}^i_k$, $\bm{Q}^i_{k+1}$ and $\bm{Q}^b_k$, and Green's formula to get
 \begin{eqnarray*}
(\nabla_{w,k+1}\cdot\{\bm{Q}^i_{k}\bm{\tau}, \bm{Q}^b_{k+1}\bm{\tau}\},\bm{v}_{hi})_T&=&
-(\bm{Q}^i_{k}\bm{\tau},\nabla\bm{v}_{hi})_T+\langle  \bm{Q}^b_{k+1}\bm{\tau}\bm{n},\bm{v}_{hi} \rangle_{\partial T}
\nonumber\\
           &=&-(\bm{\tau},\nabla\bm{v}_{hi})_T+\langle  \bm{\tau}\bm{n},\bm{v}_{hi} \rangle_{\partial T}
\nonumber\\
     &=& ( \nabla\cdot\bm{\tau},\bm{v}_{hi})_T \nonumber\\
      &=& ( \bm{Q}^i_{k+1} \nabla\cdot\bm{\tau},\bm{v}_{hi})_T,
 \end{eqnarray*}
 which implies $\nabla_{w,k+1}\cdot\{\bm{Q}^i_{k}\bm{\tau}, \bm{Q}^b_{k+1}\bm{\tau}\}= \bm{Q}^i_{k+1} \nabla\cdot\bm{\tau}$.
\color{black}

From the definition of weak gradient operator $\nabla_{w,k}$, the projection property of $\bm{Q}^i_{k+1}$, and integration by parts, it follows, for all 
$\bm v\in [H^1(\Omega)]^d$, $\bm{\tau}_{hi}\in [\mathbb{P}_k(T)]^{\color{black}d\times d\color{black}}$,
\begin{eqnarray*}
(\nabla_{w,k}\{\bm{Q}^i_{k+1}\bm{v}, \bm{\mathcal{I}}_{k+1}\bm{v}\color{black}\},\bm{\tau}_{hi})_T&=&
-(\bm{Q}^i_{k+1}\bm{v},\nabla\cdot\bm{\tau}_{hi})_T+\langle \bm{\mathcal{I}}_{k+1}\bm{v},\bm{\tau}_{hi}\bm{n}\rangle_{\partial T}\nonumber\\
 &=&
-(\bm{v},\nabla\cdot\bm{\tau}_{hi})_T+\langle \bm{v},\bm{\tau}_{hi}\bm{n}\rangle_{\partial T}\nonumber\\
&&\quad +\langle \bm{\mathcal{I}}_{k+1}\bm{v}-\bm{v},\bm{\tau}_{hi}\bm{n}\rangle_{\partial T}\nonumber\\
&=&(\nabla\bm{v},\bm{\tau}_{hi})_T+\langle \bm{\mathcal{I}}_{k+1}\bm{v}-\bm{v},\bm{\tau}_{hi}\bm{n}\rangle_{\partial T}.
 \end{eqnarray*}
This proves \eqref{pwg1}.

For   $\bm{v}_h=\{\bm{v}_{hi},\bm{v}_{hb}\}\in\bm{U}_h$, we apply integration by parts,  the definitions of $\bm{\epsilon}_h(\bm{v}_{hi}),\ \bm{\epsilon}_{w,k}(\bm{v}_{h}),$ and $\nabla_{w,k} \bm{v}_{h}$ to get
\begin{eqnarray*}
(\bm{\epsilon}_h(\bm{v}_{hi}),\bm{\epsilon}_h(\bm{v}_{hi}))&=&-(\bm{v}_{hi}
,\nabla_h\cdot\bm{\epsilon}_h(\bm{v}_{hi}))
+\langle\bm{\epsilon}_h(\bm{v}_{hi})\bm{n},\bm{v}_{hi}\rangle_{\partial \mathcal{T}_h}\nonumber\\
 &=&(\bm{\epsilon}_{w,k}(\bm{v}_{hi})
,\bm{\epsilon}_h(\bm{v}_{hi}))
+\langle\bm{\epsilon}_h(\bm{v}_{hi})\bm{n},\bm{v}_{hi}-\bm{v}_{hb}\rangle_{\partial \mathcal{T}_h}.
\end{eqnarray*}
As a result,  the estimate (\ref{pwg3}) follows from  Cauchy-Schwarz inequality and the  inverse inequality. 
\end{proof}

Set
\begin{eqnarray*}
\bm\Sigma:=\left\{ \bm\tau\in [L^2(\Omega)]^{d\times d}:\ \bm\tau^T=\bm\tau\right\},\quad
\bm U:=\left\{ \bm  v\in [H^1(\Omega)]^{d}:\ \bm v|_{\Gamma_D}=\bm{g}_D\right\}.
\end{eqnarray*}

\begin{lemma}\label{lemma42} Let $(\bm{\sigma},\bm{u})\in (\bm\Sigma\cap \bm H(div,\Omega))\color{black}\times \bm U
$
be the solution to the model $(\ref{o1})$.  Then, for all $\bm{\tau}_{hi}\in \bm{\Sigma}_{hi} $, $ \tau_{hb}^{tr}\in \Sigma_{hb}^{tr}$,  and $ \bm{v}_h=\{\bm{v}_{hi},\bm{v}_{hb}\}\in \bm{U}_{hi}\times\bm{U}_{hb}^{D,\bm{0}}$, it holds the following error equations:
\begin{eqnarray}\label{e1101}
a_h(\bm{Q}^i_k\bm{\sigma},\bm{\tau}_{hi})
-b_h(\bm{\tau}_{hi},\{\bm{Q}^i_{k+1}\bm{u},\bm{\mathcal{I}}_{k+1}\bm{u}\})= E_0( \bm{u},\bm{\tau}_{hi})
 \end{eqnarray}
for $  k+1\ge d$,
\begin{eqnarray}
a_h(\bm{Q}^i_k\bm{\sigma},\bm{\tau}_{hi})
+{z}_h(\bm{Q}^i_k\bm{\sigma},\mathcal{I}_1tr(\bm{\sigma});\bm{\tau}_{hi},\tau^{tr}_{hb})&& \nonumber\\
-b_h(\bm{\tau}_{hi},\{\bm{Q}^i_{k+1}\bm{u},\bm{\mathcal{I}}_{k+1}\bm{u}\})&=&E_1(\bm{\sigma},\bm{u};\bm{\tau}_{hi},\tau^{tr}_{hb})
\label{e1100}
 \end{eqnarray}
 for $  k+1< d$,
and
\begin{eqnarray}\label{e120}
-b_h(\bm{Q}^i_k\bm{\sigma},\bm{v}_{h})-s_h(\{\bm{Q}^i_{k+1}\bm{u},\bm{\mathcal{I}}_{k+1}\bm{u}\},\bm{v}_h)
=(\bm{f},\bm{v}_{hi})-\langle \bm{g}_N,\bm{v}_{hb}\rangle_{\Gamma_N}+E_2(\bm{\sigma},\bm{u};\bm{v}_{h}),
 \end{eqnarray}
 where
 \begin{eqnarray}\label{E-form}
 E_0( \bm{u},\bm{\tau}_{hi})&:=&
-\langle \bm{\mathcal{I}}_{k+1}\bm{u}-\bm{u},\bm{\tau}_{hi}\bm{n}\rangle_{\partial \mathcal{T}_h},\\
 E_1(\bm{\sigma},\bm{u};\bm{\tau}_{hi},\tau^{tr}_{hb})&:=& -\langle \bm{\mathcal{I}}_{k+1}\bm{u}-\bm{u},\bm{\tau}_{hi}\bm{n}\rangle_{\partial \mathcal{T}_h}\nonumber\\
&&+\langle\beta(tr(\bm{Q}_k^i\bm{\sigma})-\mathcal{I}_1tr(\bm{\sigma})), tr(\bm{\tau}_{hi})-\tau^{tr}_{hb} \rangle_{\partial\mathcal{T}_h/\partial\Omega},\\
 E_2(\bm{\sigma},\bm{u};\bm{v}_{h})&:=& -\langle\alpha(\bm{Q}^i_{k+1}\bm{u}-\bm{\mathcal{I}}_{k+1}\bm{u}),\bm{v}_{hi}-\bm{v}_{hb}\rangle_{\partial\mathcal{T}_h}\nonumber\\
                                 &&-\langle\bm{\sigma}\bm{n}- \bm{Q}^i_k\bm{\sigma}\bm{n},\bm{v}_{hi}-\bm{v}_{hb}\rangle_{\partial\mathcal{T}_h}.\label{E-form-E2}
\end{eqnarray}

\end{lemma}
\begin{proof}By the definitions of $a_h(\cdot,\cdot)$, $b_h(\cdot,\cdot)$, the property (\ref{pwg1}) and the relation $\mathcal{A}\bm{\sigma}-\bm{\epsilon}(\bm{u})=\bm{0}$ in (\ref{o1}),  we   get
\begin{eqnarray}
&&a_h(\bm{Q}^i_k\bm{\sigma},\bm{\tau}_{hi})
-b_h(\bm{\tau}_{hi},\{\bm{Q}^i_{k+1}\bm{u},\bm{\mathcal{I}}_{k+1}\bm{u}\})\nonumber\\
&&\quad\quad=(\mathcal{A}\bm{\sigma}-\nabla\bm{u},\bm{\tau}_{hi})-\langle \bm{\mathcal{I}}_{k+1}\bm{u}-\bm{u},\bm{\tau}_{hi}\bm{n}\rangle_{\partial \mathcal{T}_h}
\nonumber\\
&&\quad\quad=(\mathcal{A}\bm{\sigma}-\bm{\epsilon}(\bm{u}),\bm{\tau}_{hi})-\langle \bm{\mathcal{I}}_{k+1}\bm{u}-\bm{u},\bm{\tau}_{hi}\bm{n}\rangle_{\partial \mathcal{T}_h}
\nonumber\\
&&\quad\quad=-\langle \bm{\mathcal{I}}_{k+1}\bm{u}-\bm{u},\bm{\tau}_{hi}\bm{n}\rangle_{\partial \mathcal{T}_h},
 \end{eqnarray}
 i.e. \eqref{e1101} holds.

 The relation  \eqref{e1100} follows from \eqref{e1101} and
   the definition of
$z_h(\cdot,\cdot;\cdot,\cdot)$.

The thing left is to show  \eqref{e120}. From integration by parts and the definitions of the weak gradient $\nabla_{w,k}$ and the projection $\bm{Q}^i_k$, it follows
\begin{eqnarray*}
-b_h(\bm{Q}^i_k\bm{\sigma},\bm{v}_{h})&=&(\nabla_h\cdot\bm{Q}^i_k\bm{\sigma},\bm{v}_{hi})-\langle \bm{Q}^i_k\bm{\sigma}\bm{n},\bm{v}_{hb}\rangle_{\partial\mathcal{T}_h}\nonumber\\
&=&-(\bm{Q}^i_k\bm{\sigma},\nabla_h\bm{v}_{hi})-\langle \bm{Q}^i_k\bm{\sigma}\bm{n},\bm{v}_{hb}-\bm{v}_{hi}\rangle_{\partial\mathcal{T}_h},\\
&=&(\nabla\cdot\bm{\sigma},\bm{v}_{hi})-\langle \bm{\sigma}\bm{n},\bm{v}_{hi}\rangle_{\partial\mathcal{T}_h}-\langle \bm{Q}^i_k\bm{\sigma}\bm{n},\bm{v}_{hb}-\bm{v}_{hi}\rangle_{\partial\mathcal{T}_h},
 \end{eqnarray*}
which, together with   (\ref{o1}), the relation $\langle \bm{\sigma}\bm{n},\bm{v}_{hb}\rangle_{\partial\mathcal{T}_h}=\langle \bm{\sigma}\bm{n},\bm{v}_{hb}\rangle_{\Gamma_N}$ and the definition of $s_h(\cdot,\cdot)$, yields the relation \eqref{e120}.
\end{proof}

Introduce the space of (infinitesimal) rigid motions on $\Omega$:
 \begin{eqnarray}
\bm{RM}(\Omega)=\{\bm{a}+\bm{b}\bm{\eta}:\bm{a}\in\mathbb{R}^d,\bm{\eta}\in \mathfrak{so}(d)\},
\end{eqnarray}
where $\mathfrak{so}(d)$ is the Lie algebra of anti-symmetric $d\times d$ matrices. The space $\bm{RM}(\Omega)$ is precisely the kernel of the strain tensor.
We  recall the Piecewise Korn's inequality as follows.

\begin{lemma} \label{lemmaKorn}\cite{Korn} (Piecewise Korn's inequality) Let $\mathcal{T}_h$  be a  shape-regular decomposition of $\Omega$, then for any  $\bm{v}_p\in [H^1(\mathcal{T}_h)]^d$ it holds 
\begin{eqnarray}
\|\nabla_h\bm{v}_p\|^2_0\lesssim \|\bm{\epsilon}_h(\bm{v}_p)\|^2_0+\sup_{\substack{\bm{m}\in\bm{RM}(\Omega), \\ \|\bm{m}\|_{0,\Gamma_s}=1, \\ \int_{\Gamma_s}\bm{m}\ ds=\bm{0}}}\langle \bm{v}_p\cdot\bm{m} \rangle_{\Gamma_s}^2+\sum_{E\in\mathcal{E}_h/\partial\Omega}h_E^{-1}\|\bm{Q}^b_1[\bm{v}_p]\|^2_{0,E}.
\end{eqnarray}
where 
$\Gamma_s$ is a measurable subset of $\partial\Omega$ with a positive $d-1$-dimensional, and  $[\bm{v}_p]|_E$   denote the jump   of $\bm{v}_p$ over $E\in\mathcal{E}_h$.
\end{lemma}

By Lemma \ref{lemmaKorn}, we can prove the following lemma.
\begin{lemma}
 \label{lemmaKorn2}
For any  $\bm{v}_h=\{\bm{v}_{hi},\bm{v}_{hb}\}\in \bm{U}_{h}^{\bm{0}}$, 
it holds
\begin{eqnarray}
\|\nabla_h\bm{v}_{hi}\|^2_0\lesssim\|\bm{\epsilon}_h(\bm{v}_{hi})\|^2_0+\sum_{T\in\mathcal{T}_h}h_T^{-1}\|\bm{v}_{hi}-\bm{v}_{hb}\|^2_{0,\partial T}. \label{41500}
\end{eqnarray}
\end{lemma}
\begin{proof} For $\bm{v}_h=\{\bm{v}_{hi},\bm{v}_{hb}\}\in \bm{U}_{h}^{\bm{0}}$, we apply Lemma \ref{lemmaKorn}, Cauchy-Schwarz inequality to get
\begin{eqnarray}
\|\nabla_h\bm{v}_{hi}\|^2_0&\lesssim& \|\bm{\epsilon}_h(\bm{v}_{hi})\|^2_0+\|\bm{v}_{hi}\|^2_{0,\Gamma_D}+\sum_{E\in\mathcal{E}_h/\partial\Omega}h_E^{-1}\|\bm{Q}^b_1[\bm{v}_{hi}]\|^2_{0,E}.
\nonumber\\
&\le&\|\bm{\epsilon}_h(\bm{v}_{hi})\|^2_0+\|\bm{v}_{hi}\|^2_{0,\Gamma_D}+\sum_{E\in\mathcal{E}_h/\partial\Omega}h_E^{-1}\|\bm{Q}^b_k[\bm{v}_{hi}]\|^2_{0,E}.
\nonumber\\
&=&\|\bm{\epsilon}_h(\bm{v}_{hi})\|^2_0+\|\bm{v}_{hi}-\bm{v}_{hb}\|^2_{0,\Gamma_D}+\sum_{E\in\mathcal{E}_h/\partial\Omega}h_E^{-1}\|[\bm{v}_{hi}-\bm{v}_{hb}]\|^2_{0,E}.
\nonumber\\
&\lesssim&\|\bm{\epsilon}_h(\bm{v}_{hi})\|^2_0+\sum_{T\in\mathcal{T}_h}h_T^{-1}\|\bm{v}_{hi}-\bm{v}_{hb}\|^2_{0,\partial T}.
\end{eqnarray}
This completes the proof.
\end{proof}

\begin{lemma}\label{lemma43} For   $(\bm{\sigma},\bm{u})\in(\bm\Sigma\bigcap [H^{k+1}(\Omega)]^{d\times d})\times (\bm U\bigcap [H^{k+2}(\Omega)]^{d})$ and $\bm{v}_h=\{\bm{v}_{hi},\bm{v}_{hb}\}\in\bm{U}_{h}^{\bm{0}}$, it holds
 \begin{eqnarray}
  | E_0(\bm{u},\bm{\tau}_{hi})|
 \lesssim h^{k+1}|\bm{u}|_{k+2}\|\bm{\tau}_{hi}\|_0,\label{42001}\\
 | E_2(\bm{\sigma},\bm{u};\bm{v}_{h})|\lesssim h^{k+1}\big(\mu^{-1/2}|\bm{\sigma}|_{k+1}+\mu^{1/2}|\bm{u}|_{k+2}\big) \| \alpha^{1/2}(\bm{v}_{hi}-\bm{v}_{hb})\|_{\partial\mathcal{T}_h},
\label{42002}
  \end{eqnarray}
  and, for $k+1<d$, 
 \begin{eqnarray}\label{42003}
 | E_1(\bm{\sigma},\bm{u};\bm{\tau}_{hi},\tau^{tr}_{hb})|
 &\lesssim& h^{k+1}\big(\mu^{-1/2}|\bm{\sigma}|_{k+1}+\mu^{1/2}|\bm{u}|_{k+2}\big)\nonumber\\
 &&\cdot\big(\mu^{-1/2}\|\bm{\tau}_{hi}\|_0+\|\beta^{1/2}( tr(\bm{\tau}_{hi})-\tau^{tr}_{hb} )\|_{\partial\mathcal{T}_h/\partial\Omega}\big).
 \end{eqnarray}

\end{lemma}
\begin{proof} The estimates \eqref{42001}-\eqref{42002}  follow from  Cauchy-Schwarz inequality, the inverse inequality,   Lemmas \ref{lemma23} and  \ref{lemma2.6}.

Similarly,  we have
\begin{eqnarray}
| E_1(\bm{\sigma},\bm{u};\bm{\tau}_{hi},\tau^{tr}_{hb})|&\lesssim& h^{k+1}|\bm{u}|_{k+2}\|\bm{\tau}_{hi}\|_0\nonumber\\
&&+\mu^{-1/2}h^{2}|\bm{\sigma}|_{k+1}\|\beta^{1/2}( tr(\bm{\tau}_{hi})-\tau^{tr}_{hb} )\|_{\partial\mathcal{T}_h/\partial\Omega}\nonumber\\
&\lesssim& h^{k+1}|\bm{u}|_{k+2}\|\bm{\tau}_{hi}\|_0\nonumber\\
&&+\mu^{-1/2}h^{k+1}|\bm{\sigma}|_{k+1}\|\beta^{1/2}( tr(\bm{\tau}_{hi})-\tau^{tr}_{hb} )\|_{\partial\mathcal{T}_h/\partial\Omega},
\end{eqnarray}
where we have used the the fact  $h^2\le h^{k+1}$  for  $h<1$ and $k+1<d$ with $d=2,3$. This completes the proof.

\end{proof}

%
%

\begin{lemma} \label{lemma44} Let  $(\bm{\sigma},\bm{u})\in(\bm\Sigma\bigcap [H^{k+1}(\Omega)]^{d\times d})\times (\bm U\bigcap [H^{k+2}(\Omega)]^{d})$  be the solution to the model $(\ref{o1})$, and let
$\left(\bm{\sigma}_{hi}, \sigma_{hb}^{tr}, \bm{u}_h:=\{\bm{u}_{hi},\bm{u}_{hb}\}\right)\in  \bm{\Sigma}_{hi}\times \Sigma_{hb}^{tr}\times \bm{U}_{h}^{\bm{g}_D}$ and $\left(\bm{\sigma}_{hi},  \bm{u}_h\right)\in  \bm{\Sigma}_{hi}\times \bm{U}_{h}^{\bm{g}_D}$ be the solutions to the
 WG schemes \eqref{wg1new}-\eqref{wg2new} and \eqref{wg1}-\eqref{wg2}, respectively.
%
 Then it holds 
\begin{eqnarray}\label{4.19}
&&a_h(\bm{\xi}^{\sigma}_{hi},\bm{\xi}^{\sigma}_{hi})
+{z}_h\left(tr(\bm{\xi}^{\sigma}_{hi}),\xi^{tr}_{hb};tr(\bm{\xi}^{\sigma}_{hi}),\xi^{tr}_{hb}\right)
+s_h(\bm{\xi}^u_h,\bm{\xi}^u_h)\nonumber\\
&&\quad\quad\quad\quad\quad\lesssim h^{2k+2}(\mu^{-1}|\bm{\sigma}|^2_{k+1}+\mu|\bm{u}|^2_{k+2})
\end{eqnarray}
 for $k+1<d$, and
\begin{eqnarray}\label{4.199}
a_h(\bm{\xi}^{\sigma}_{hi},\bm{\xi}^{\sigma}_{hi})
+s_h(\bm{\xi}^u_h,\bm{\xi}^u_h)\lesssim h^{2k+2}(\mu^{-1}|\bm{\sigma}|^2_{k+1}+\mu|\bm{u}|^2_{k+2})
\end{eqnarray}
for $k+1\ge d$, where \color{black}
\begin{eqnarray}\label{xi-sigma-u}\small
\bm{\xi}_{hi}^{\sigma}:=\bm{Q}^i_{k}\bm{\sigma}-\bm{\sigma}_{hi}, \
\xi^{tr}_{hb}:= \mathcal{I}_{1}tr(\bm{\sigma})-\sigma_{hb}^{tr}, \
\bm{\xi}_{h}^u:=\{\bm{\xi}_{hi}^u,\bm{\xi}_{hb}^u\}
\end{eqnarray}
   with
   $$\bm{\xi}_{hi}^u:=\bm{Q}^i_{k+1}\bm{u}-\bm{u}_{hi} ,\quad
\bm{\xi}_{hb}^u:=\bm{\mathcal{I}}_{k+1}\bm{u}-\bm{u}_{hb}.
$$
\end{lemma}
\begin{proof}
From  the relations (\ref{wg1new})-(\ref{wg2}) and (\ref{e1101})-(\ref{e120}), we have,
 for all $\bm{\tau}_{hi}\in \bm{\Sigma}_{hi} $,  $\tau^{tr}_{hb}\in\Sigma_{hb}^{tr}$, and $ \bm{v}_h=\{\bm{v}_{hi},\bm{v}_{hb}\}\in \bm{U}_{h}^{\bm{0}}$,
\begin{eqnarray}\footnotesize
&&a_h(\bm{\xi}_{hi}^{\sigma},\bm{\tau}_{hi})
+{z}_h(\bm{\xi}^{\sigma}_{hi},\xi^{tr}_{hb};\bm{\tau}_{hi},\tau^{tr}_{hb})
-b_h(\bm{\tau}_{hi},\bm{\xi}^u_h)\nonumber\\
&&\hskip3cm= E_1(\bm{\sigma},\bm{u};\bm{\tau}_{hi},\tau^{tr}_{hb})\quad \text{if } k+1<d,\label{e11-21}\\
&&a_h(\bm{\xi}_{hi}^{\sigma},\bm{\tau}_{hi})
-b_h(\bm{\tau}_{hi},\bm{\xi}^u_h)= E_0(\bm{u};\bm{\tau}_{hi})\quad \text{if } k+1\ge d,\label{e12-22}\\
&& -b_h(\bm{\xi}_{hi}^{\sigma},\bm{v}_{h})-s_h(\bm{\xi}^u_h,\bm{v}_h)= E_2(\bm{\sigma},\bm{u};\bm{v}_{h}).\label{e12}
 \end{eqnarray}
Then the desired estimates \eqref{4.19}-\eqref{4.199} follow from    Lemmas \ref{lemmaKorn2}-\ref{lemma43}  and Young's inequality.
\end{proof}

\begin{lemma}\label{lemma45} Under the conditions of Lemma \ref{lemma44}, let  $\bm{\sigma}_{hb}\in [L^2(\partial \mathcal{T}_h)]^{d\times d}$ satisfy
 \begin{eqnarray} \label{defhb}
(\bm{\sigma}_{hb}\bm n)|_{\partial T}:=(\bm{\sigma}_{hi}\bm{n})|_{\partial T}
                                       -\alpha(\bm{u}_{hi}|_{\partial T}-\bm{u}_{hb}), \quad \forall T\in \mathcal{T}_h.
  \end{eqnarray}
Then
it holds
 \begin{eqnarray}
 &&\nabla_{w,\color{black}k+1\color{black}}\cdot\{\bm{Q}^i_k\bm{\sigma}-\bm{\sigma}_{hi}, \color{black}\bm{{Q}}^b_{k+1}\color{black}\bm{\sigma} -\bm{\sigma}_{hb}\}=\bm{0},\label{AS1}\\
&&\langle \color{black}\bm{Q}^b_{k+1}\color{black}(\bm{\sigma}\bm{n}) -\bm{\sigma}_{hb}\bm n,\bm{v}_{hb}\rangle_{\partial\mathcal{T}_h}=0,
 \forall \bm{v}_{hb}\in \bm{U}^{D,\bm{0}}_{hb},\label{AS2}\\
 &&tr(\bm{Q}^i_k\bm{\sigma}-\bm{\sigma}_{hi})\in L^2_0(\Omega) \text{ if }\Gamma_N=\emptyset.\label{AS3}
  \end{eqnarray}
\end{lemma}
\begin{proof}
We first show \eqref{AS1}.  Taking $\bm{v}_{hi}=\bm{0}$ and $\bm{v}_{hb}\in \bm{U}^{D,\bm{0}}_{hb}$ in \eqref{wg2new} or \eqref{wg2},  we  get
\begin{eqnarray}
-\langle\bm{\sigma}_{hi}\bm{n},\bm{v}_{hb} \rangle_{\partial\mathcal{T}_h}
+\langle\alpha( \bm{u}_{hi}-\bm{u}_{hb} ),\bm{v}_{hb}\rangle_{\partial\mathcal{T}_h}=-\langle \bm{g}_N,\bm{v}_{hb}\rangle_{\Gamma_N},
 \end{eqnarray}
 which, together with  \eqref{defhb},  yields
 \begin{eqnarray}
 \langle\bm{\sigma}_{hb}\bm n,\bm{v}_{hb}\rangle_{\partial\mathcal{T}_h}=\langle \bm{g}_N,\bm{v}_{hb}\rangle_{\Gamma_N}.\label{hb}
 \end{eqnarray}
On one hand, for all $\color{black}\bm{v}_{hi}\in\bm{U}_{hi}\color{black}$, $\bm{v}_{hb}\in \bm{U}^{D,\bm{0}}_{hb}$,  by the definitions of the discrete weak divergence and weak gradient,  we have
\begin{eqnarray*}
(\nabla_{w,\color{black}k+1\color{black}}\cdot\{\bm{\sigma}_{hi},\bm{\sigma}_{hb}\},\bm{v}_{hi})
&=&-(\bm{\sigma}_{hi},\nabla_h\bm{v}_{hi})+\langle \bm{\sigma}_{hb}\bm n,\bm{v}_{hi}\rangle_{\partial\mathcal{T}_h}\nonumber\\
&=&(\nabla_h\cdot\bm{\sigma}_{hi},\bm{v}_{hi})+\langle \bm{\sigma}_{hb}\bm n-\bm{\sigma}_{hi}\bm{n},\bm{v}_{hi}\rangle_{\partial\mathcal{T}_h}\nonumber\\
&=&-(\bm{\sigma}_{hi},\nabla_{w,k}\bm{v}_{hi})+\langle \bm{\sigma}_{hb}\bm n-\bm{\sigma}_{hi}\bm{n},\bm{v}_{hi}\rangle_{\partial\mathcal{T}_h}
\nonumber\\
&&+\langle \bm{\sigma}_{hi}\bm{n},\bm{v}_{hb}\rangle_{\partial\mathcal{T}_h},
 \end{eqnarray*}
which, together with (\ref{hb}), \eqref{defhb}, and (\ref{wg2}), implies
\begin{eqnarray}
(\nabla_{w,\color{black}k+1\color{black}}\cdot\{\bm{\sigma}_{hi},\bm{\sigma}_{hb}\},\bm{v}_{hi})
&=&-(\bm{\sigma}_{hi},\nabla_{w,k}\bm{v}_{hi})+\langle \color{black}\bm{\sigma}_{hb}\bm{n}\color{black}-\bm{\sigma}_{hi}\bm{n},\bm{v}_{hi}-\bm{v}_{hb}\rangle_{\partial\mathcal{T}_h}\nonumber\\
&&+\langle \bm{g}_N,\bm{v}_{hb}\rangle_{\Gamma_N}
\nonumber\\
&=&-(\bm{\sigma}_{hi},\nabla_{w,k}\bm{v}_{hi})-\langle \alpha(\bm{u}_{hi}-\bm{u}_{hb}),\bm{v}_{hi}-\bm{v}_{hb}\rangle_{\partial\mathcal{T}_h}\nonumber\\
&&+\langle \bm{g}_N,\bm{v}_{hb}\rangle_{\Gamma_N}
\nonumber\\
&=&(\bm{f},\bm{v}_{hi}),\label{43200}
\end{eqnarray}
On the other hand,
from the  commutativity \color{black} property (\ref{pwg2}) and the fact $\nabla\cdot\bm{\sigma}=\bm{f}$ in (\ref{o1}), it follows
\begin{eqnarray*}
(\nabla_{w,\color{black}k+1\color{black}}\cdot\{\bm{Q}^i_k\bm{\sigma},\bm{Q}^b_{\color{black}k+1\color{black}}\bm{\sigma}\},\bm{v}_{hi})
=(\bm{Q}^i_{\color{black}k+1\color{black}}\nabla\cdot\bm{\sigma},\bm{v}_{hi})=(\nabla\cdot\bm{\sigma},\bm{v}_{hi})=(\bm{f},\bm{v}_{hi}),
\end{eqnarray*}
which, together with  (\ref{43200}), yields (\ref{AS1}).

 By \eqref{hb} and the boundary condition  $\bm{\sigma}\bm{n} =\bm{g}_N$  on  $\Gamma_N$ we easily get
\begin{eqnarray*}
\langle\color{black}\bm{Q}^b_{k+1}\color{black}(\bm{\sigma}\bm{n}) -\bm{\sigma}_{hb}\bm{n},\bm{v}_{hb}\rangle_{\partial\mathcal{T}_h}=
\langle\bm{\sigma}\bm{n},\bm{v}_{hb}\rangle_{\partial\mathcal{T}_h}-\langle \bm{g}_N,\bm{v}_{hb}\rangle_{\Gamma_N}
=0,
\end{eqnarray*}
i.e.     \eqref{AS2} holds.

The work left is to prove \eqref{AS3}. When $\Gamma_N=\emptyset$, we use the first relation in \eqref{o1} to get
\begin{eqnarray}\small
\frac{1}{2\mu+d\lambda}(tr(\bm{Q}^i_k\bm{\sigma}),1)=\frac{1}{2\mu+d\lambda}(tr(\bm{\sigma}),1)=(\mathcal{A}\bm{\sigma},I)=(\bm{\epsilon}(\bm{u}),{I})=\langle \bm{g}_D,\bm{n} \rangle_{\partial\Omega}.\label{429}
\end{eqnarray}
Since $$ (I,\nabla_{w,k}\bm{u}_{h})= \langle \bm{\mathcal{I}}_{k+1}\bm{g}_D,\bm{n} \rangle_{\partial\Omega},$$
 taking $\bm{\tau}_{hi}=I$ in \eqref{wg1}  yields
\begin{eqnarray*}
\frac{1}{2\mu+d\lambda}(tr(\bm{\sigma}_{hi}),1)&=&(\mathcal{A}\bm{\sigma}_{hi},{I})
 =\langle \bm{\mathcal{I}}_{k+1}\bm{g}_D,\bm{n} \rangle_{\partial\Omega},
 \end{eqnarray*}
which, together with  \eqref{2.5}, leads to the desired relation  \eqref{AS3} for  the case of $k+1\ge d$.

For the case of $k+1<d$,   we take $\bm{\tau}_{hi}=I$ in \eqref{wg1new}   to get
\begin{eqnarray}
\frac{1}{2\mu+d\lambda}(tr(\bm{\sigma}_{hi}),1)&=&(\mathcal{A}\bm{\sigma}_{hi},{I}) \nonumber\\
  &=&\langle \bm{\mathcal{I}}_{k+1}\bm{g}_D,\bm{n} \rangle_{\partial\Omega} \nonumber\\
  && \quad
-\langle\beta (tr(\bm{\sigma}_{hi})-\sigma^{tr}_{hb}),   d- \tau^{tr}_{hb}
  \rangle_{\partial\mathcal{T}_h/\partial\Omega},
\end{eqnarray}
where we have used the fact that  $tr(I)=d$.
 On the other hand, taking
$\bm{\tau}_{hi}=\bm{0}$ in \eqref{wg1new} yields
\begin{eqnarray*}
\langle\beta (tr(\bm{\sigma}_{hi})-\sigma^{tr}_{hb}),\tau^{tr}_{hb}
  \rangle_{\partial\mathcal{T}_h/\partial\Omega}=0,\ \forall  \tau_{hb}^{tr}\in \Sigma_{hb}^{tr},
\end{eqnarray*}
which means
\begin{eqnarray}
\langle\beta (tr(\bm{\sigma}_{hi})-\sigma^{tr}_{hb}),d-\tau^{tr}_{hb}
  \rangle_{\partial\mathcal{T}_h/\partial\Omega}=0,\label{430}
\end{eqnarray}
since $d\in\Sigma_{hb}^{tr}$.
As a result,
  \eqref{AS3} follows from \eqref{2.5} and  \eqref{429}-\eqref{430}.

%
\end{proof}

Finally, we shall derive the following  error estimates for the stress and displacement approximations.

\begin{theorem} \label{th48} Let  $(\bm{\sigma},\bm{u})\in(\bm\Sigma\bigcap [H^{k+1}(\Omega)]^{d\times d})\times (\bm U\bigcap [H^{k+2}(\Omega)]^{d})$  be the solution to the model $(\ref{o1})$, and let
$\left(\bm{\sigma}_{hi}, \sigma_{hb}^{tr}, \bm{u}_h:=\{\bm{u}_{hi},\bm{u}_{hb}\}\right)\in  \bm{\Sigma}_{hi}\times \Sigma_{hb}^{tr}\times \bm{U}_{h}^{\bm{g}_D}$ and $\left(\bm{\sigma}_{hi},  \bm{u}_h\right)\in  \bm{\Sigma}_{hi}\times \bm{U}_{h}^{\bm{g}_D}$ be the solutions to the
 WG schemes \eqref{wg1new}-\eqref{wg2new} and \eqref{wg1}-\eqref{wg2}, respectively. Then it holds the following error estimates:
\begin{eqnarray}
\|\bm{\sigma}-\bm{\sigma}_{hi}\|_0&\lesssim&  h^{k+1}(|\bm{\sigma}|_{k+1}+\mu|\bm{u}|_{k+2}), \label{434}\\
\|\nabla\bm{u}-\nabla_h\bm{u}_{hi}\|_0&\lesssim& h^{k+1}(\mu^{-1}|\bm{\sigma}|_{k+1}+|\bm{u}|_{k+2} )\label{in47g}.
\end{eqnarray}
\end{theorem}
\begin{proof} Let $\bm{\sigma}_{hb}$ be  the same as   in  (\ref{defhb}),  and set
$$
\bm{\xi}_{hb}^{\sigma}:=\bm{Q}^b_{\color{black}k+1\color{black}}(\bm{\sigma}\bm{n})\otimes\bm{n}-\bm{\sigma}_{hb},$$
where
\begin{eqnarray*}
 \bm{s}\otimes\bm{r}:=\left(
   \begin{aligned}
s_1r_1&&\cdots&&s_1r_d\\
\vdots&& &&\vdots\\
s_dr_1&&\cdots&&s_dr_d
  \end{aligned}
\right) \quad \text{ for } \bm s=(s_1,\cdots,s_d)^T, \bm r=(r_1,\cdots,r_d)^T.
 \end{eqnarray*}
  Recall from \eqref{xi-sigma-u} that $$\bm{\xi}_{hi}^{\sigma}=\bm{Q}^i_{k}\bm{\sigma}-\bm{\sigma}_{hi},\quad
\xi^{tr}_{hb}= \mathcal{I}_{1}tr(\bm{\sigma})-\sigma_{hb}^{tr}, \
\bm{\xi}_{h}^u=\{\bm{Q}^i_{k+1}\bm{u}-\bm{u}_{hi},\bm{\mathcal{I}}_{k+1}\bm{u}-\bm{u}_{hb}\}.$$
Then, in view of  
 \eqref{227}, Lemma \ref{lemma23}, and Lemma \ref{lemma2.6}, we have
\begin{eqnarray}
&&\|\alpha^{-1/2}(\bm{\xi}^{\sigma}_{hi}\bm{n}- \bm{\xi}^{\sigma}_{hb}\bm n)\|^2_{0,\partial\mathcal{T}_h}\nonumber\\
&&\quad\quad= \|\alpha^{-1/2}(\bm{Q}^i_k\bm{\sigma}\bm{n}-\bm{Q}^b_{\color{black}k+1\color{black}}(\bm{\sigma}\bm{n}))-\alpha^{-1/2}(\bm{\sigma}_{hi}\bm{n}-\bm{\sigma}_{hb}\bm n)\|^2_{0,\partial\mathcal{T}_h}
\nonumber\\
&&\quad\quad= \|\alpha^{-1/2}(\bm{Q}^i_k\bm{\sigma}\bm{n}-\bm{Q}^b_{\color{black}k+1\color{black}}(\bm{\sigma}\bm{n}))-\alpha^{1/2}(\bm{u}_{hi}-\bm{u}_{hb})\|^2_{0,\partial\mathcal{T}_h}
\nonumber\\
&&\quad\quad= \|\alpha^{-1/2}(\bm{Q}^i_k\bm{\sigma}\bm{n}-\bm{Q}^b_{\color{black}k+1\color{black}}(\bm{\sigma}\bm{n}))-\alpha^{1/2}(\bm{u}_{hi}-\bm{u}_{hb})\|^2_{0,\partial\mathcal{T}_h}
\nonumber\\
&&\quad\quad\le \|\alpha^{1/2}\left((\bm{Q}^i_{k+1}\bm{u}-\bm{u}_{hi})-(\bm{\mathcal{I}}_{k+1}\bm{u}-\bm{u}_{hb})\right)\|^2_{0,\partial\mathcal{T}_h}\nonumber\\
&&\quad\quad\quad
+\|\alpha^{-1/2}(\bm{Q}^i_k\bm{\sigma}\bm{n}-\bm{Q}^b_{\color{black}k+1\color{black}}(\bm{\sigma}\bm{n}))\|^2_{0,\partial\mathcal{T}_h}+\|\alpha^{1/2}(\bm{Q}^i_{k+1}\bm{u}-\bm{\mathcal{I}}_{k+1}\bm{u})\|^2_{0,\partial\mathcal{T}_h}\nonumber\\
&&\quad\quad\lesssim s_h(\bm{\xi}^u_h,\bm{\xi}^u_h)+h^{2k+2}(\mu^{-1}|\bm{\sigma}|^2_{k+1}+\mu|\bm{u}|^2_{k+2}).\label{in461}
\end{eqnarray}
Thus, from Lemma \ref{lemma45} and   Theorem \ref{th34} it follows
\begin{eqnarray*}
\|\bm{\xi}^{\sigma}_{hi}\|_0^2&\lesssim\left\{\begin{array}{l}
\mu a_h(\bm{\xi}^{\sigma}_{hi},\bm{\xi}^{\sigma}_{hi})
+\mu\|\alpha^{-1/2}(\bm{\xi}^{\sigma}_{hi}\bm{n}-\bm{\xi}^{\sigma}_{hb}\bm{n})\|^2_{0,\partial\mathcal{T}_h},
+\mu\|\beta^{1/2} (tr(\bm{\xi}^{\sigma}_{hi})-\xi^{tr}_{hb}\|^2_{\partial\mathcal{T}_h/\partial\Omega},\\
\hskip7.5cm\  \text{if } k+1<d,\\
\mu a_h(\bm{\xi}^{\sigma}_{hi},\bm{\xi}^{\sigma}_{hi})
+\mu\|\alpha^{-1/2}(\bm{\xi}^{\sigma}_{hi}\bm{n}-\bm{\xi}^{\sigma}_{hb}\bm{n})\|^2_{0,\partial\mathcal{T}_h}, \qquad \text{if } k+1\geq d,
\end{array}
\right.
\end{eqnarray*}
which, together with   (\ref{in461}), Lemma \ref{lemma44}   and the approximation property of $\bm{Q}^i_{k}$, yields the desired estimate \eqref{434}.

The thing left is to  show  (\ref{in47g}). From  the relation \eqref{e11-21} 
we have, for all $\bm{\tau}_{hi}\in \bm{\Sigma}_{hi} $,
\begin{eqnarray*}
a_h(\bm{\xi}_{hi}^{\sigma},\bm{\tau}_{hi})
+{z}_h(\bm{\xi}^{\sigma}_{hi},\xi^{tr}_{hb};\bm{\tau}_{hi},0)
-b_h(\bm{\tau}_{hi},\bm{\xi}^u_h)= E_1(\bm{\sigma},\bm{u};\bm{\tau}_{hi},0)\ \text{if } k+1<d,
 \end{eqnarray*}
which, together with \eqref{e12-22},
  (\ref{pwg3}),  Theorem \ref{th31},   Theorem \ref{th34}, Lemmas \ref{lemma43}-\ref{lemma44}, and the inverse inequality, indicates
\begin{eqnarray*}
\|\bm{\epsilon}_h(\bm{\xi}^u_{hi})\|_{0} &\lesssim &\|\bm{\epsilon}_{w,k}(\bm{\xi}^u_h)\|_0+ \mu^{-1/2}\|\alpha^{1/2}(\bm{\xi}^{u}_{hi}-\bm{\xi}^{u}_{hb})\|_{0,\partial\mathcal{T}_h}  \nonumber\\
&\lesssim&  \sup_{\bm{\tau}_{hi}\in \bm{\Sigma}_{hi}}\frac{b_h(\bm{\tau}_{hi},\bm{\xi}^u_{h})}{\|\bm{\tau}_{hi}\|_0}
+\mu^{-1/2}s_h(\bm{\xi}^u_h,\bm{\xi}^u_h)^{1/2}\\
&\lesssim&  h^{k+1}(\mu^{-1}|\bm{\sigma}|_{k+1}+|\bm{u}|_{k+2}).
\end{eqnarray*}
As a result,  the desired estimate (\ref{in47g}) follows from  Lemma \ref{lemmaKorn2}, Lemma \ref{lemma44} and the triangle inequality.
\end{proof}

\section{$L^2$ error estimation for displacement approximation} \label{s5}

In order to derive the $L^2$ error estimation  for the displacement approximation $\bm{u}_{hi}$, we shall perform  Aubin-Nitsche duality argument based on the following auxiliary  problem:

Find  $\bm{\Psi}:\Omega\rightarrow  \mathbb{R}^{d\times d}_{sym}$,  $\bm{\Phi}:\Omega\rightarrow  \mathbb{R}^{d}$ such that
\begin{eqnarray}
\left\{
   \begin{aligned}
 \mathcal{A}\bm{\Psi}-\bm{\epsilon}(\bm{\Phi})&=\bm{0} &\text{ in }&\Omega, \label{o2}\\
 \nabla\cdot\bm{\Psi}&=\bm{\xi}^u_{hi} &\text{ in }&\Omega, \\
 \bm{\Phi}&=\bm{0}&\text{ on }&\Gamma_D,  \\
 \bm{\Psi}\bm{n}&=\bm{0}&\text{ on }&\Gamma_N.
  \end{aligned}
\right.
 \end{eqnarray}
Here   $\bm{\xi}_{hi}^u$ is the same as in \eqref{xi-sigma-u},  i.e. $\bm{\xi}_{hi}^u=\bm{Q}^i_{k+1}\bm{u}-\bm{u}_{hi}$.
In addition, we assume  the following regularity estimate  holds:
\begin{eqnarray}
|\bm{\Psi}|_1+\mu|\bm{\Phi}|_2\lesssim \|\bm{\xi}^u_{hi}\|_0.\label{rg}
 \end{eqnarray}

\begin{lemma}\label{lemma51}    For  $(\bm{\sigma},\bm{u})\in(\bm\Sigma\bigcap [H^{k+1}(\Omega)]^{d\times d})\times (\bm U\bigcap [H^{k+2}(\Omega)]^{d})$, it holds
\begin{eqnarray}
 |E_0(\bm{u};\bm{Q}^i_k\bm{\Psi})|&\lesssim&
 h^{k+2} |\bm{u}|_{k+2}|\bm{\Psi}|_1,\label{E0-est1}\\
 | E_0(\bm{\Phi};\bm{\xi}_{hi}^{\sigma})|&\lesssim&h^{k+2}(|\bm{\sigma}|_{k+1}+\mu|\bm{u}|_{k+2}) |\bm{\Phi}|_2   ,\label{E0-est2}\\
  |E_1(\bm{\sigma},\bm{u};\bm{Q}^i_k\bm{\Psi},\mathcal{I}_1tr(\bm{\Psi}))|&\lesssim&
 h^{k+2}(|\bm{\sigma}|_{k+1}+\mu|\bm{u}|_{k+2})\mu^{-1}|\bm{\Psi}|_1,\label{E1-est1}\\
 | E_1(\bm{\Psi},\bm{\Phi};\bm{\xi}_{hi}^{\sigma},\xi^{tr}_{hb})|&\lesssim&h^{k+2}(|\bm{\sigma}|_{k+1}+\mu|\bm{u}|_{k+2})(\mu^{-1}|\bm{\Psi}|_1+|\bm{\Phi}|_2)   ,\label{E1-est2}\\
|E_2(\bm{\sigma},\bm{u};\{\bm{Q}^i_{k+1}\bm{\Phi},\bm{\mathcal{I}}_{k+1}\bm{\Phi}\})|&\lesssim&
h^{k+2}(|\bm{\sigma}|_{k+1}+\mu|\bm{u}|_{k+2})|\bm{\Phi}|_2,\label{E1E2}\\
|E_2(\bm{\Psi},\bm{\Phi};\bm{\xi}^u_{h}))|&\lesssim&
h^{k+2}(|\bm{\sigma}|_{k+1}+\mu|\bm{u}|_{k+2})(\mu^{-1}|\bm{\Psi}|_1+|\bm{\Phi}|_2),\label{E3E4}\nonumber\\
 \end{eqnarray}
 where $ E_0(), E_1()$ and $E_2()$ are defined in \eqref{E-form}-\eqref{E-form-E2}, and
 $\bm{\xi}_{hi}^{\sigma}, \ \xi^{tr}_{hb}$ and $\bm{\xi}^u_{h}=\{\bm{\xi}_{hi}^u,\bm{\xi}_{hb}^u\}
 $
 are the same as in \eqref{xi-sigma-u}.
\end{lemma}
\begin{proof} Since the proofs of \eqref{E0-est1} and \eqref{E0-est2} follow   from those of \eqref{E1-est1} and \eqref{E1-est2}, respectively, we only prove \eqref{E1-est1}-\eqref{E3E4}.

From  $\langle \bm{\mathcal{I}}_{k+1}\bm{u}-\bm{u},\bm{Q}^b_0\bm{\Psi}\bm{n}\rangle_{\Gamma_D}=0$, $\bm{Q}^b_0\bm{\Psi}\bm{n}|_{\Gamma_N}=\bm{\Psi}\bm{n}|_{\Gamma_N}=\bm{0}$,  and Lemma  \ref{lemma2.6},   it follows
\begin{eqnarray*}\label{E_1}
 &&|E_1(\bm{\sigma},\bm{u};\bm{Q}^i_k\bm{\Psi},\mathcal{I}_1tr(\bm{\Psi}))|\le|\langle \bm{\mathcal{I}}_{k+1}\bm{u}-\bm{u},\bm{Q}^i_k\bm{\Psi}\bm{n}-\bm{Q}^b_0\bm{\Psi}\bm{n}\rangle_{\partial \mathcal{T}_h}|\nonumber\\
&&\quad\quad\quad+|\langle\beta(tr(\bm{Q}_k^i\bm{\sigma})-\mathcal{I}_1tr(\bm{\sigma})), tr(\bm{Q}^i_k\bm{\Psi})-\mathcal{I}_1tr(\bm{\Psi}) \rangle_{\partial\mathcal{T}_h/\partial\Omega}|\nonumber\\
&&\quad\quad\lesssim  h^{k+2}(|\bm{\sigma}|_{k+1}+\mu|\bm{u}|_{k+2})\mu^{-1}|\bm{\Psi}|_1,
\end{eqnarray*}
i.e.  \eqref{E1-est1} holds.
Similarly, we have
\begin{eqnarray*}\label{E_2}
 &&|E_2(\bm{\sigma},\bm{u};\{\bm{Q}^i_{k+1}\bm{\Phi},\bm{\mathcal{I}}_{k+1}\bm{\Phi}\})|
 \le|\langle\alpha(\bm{Q}^i_{k+1}\bm{u}-\bm{\mathcal{I}}_{k+1}\bm{u}),
 \bm{Q}^i_{k+1}\bm{\Phi}-\bm{\mathcal{I}}_{k+1}\bm{\Phi}\rangle_{\partial\mathcal{T}_h}
                                 |\nonumber\\
&&\quad\quad\quad+|\langle\bm{\sigma}\bm{n}- \bm{Q}^i_k\bm{\sigma}\bm{n},\bm{Q}^i_{k+1}\bm{\Phi}-\bm{\mathcal{I}}_{k+1}\bm{\Phi}\rangle_{\partial\mathcal{T}_h}|\nonumber\\
&&\quad\quad\lesssim  h^{k+2}(|\bm{\sigma}|_{k+1}+\mu|\bm{u}|_{k+2})|\bm{\Phi}|_2,
\end{eqnarray*}
i.e. \eqref{E1E2} holds.
In light of  Lemma  \ref{lemma44} and  the approximation properties of $\bm{Q}^i_k$, $\mathcal{I}_1$ and  $\bm{\mathcal{I}}_{k+1}$, we get
\begin{eqnarray}\label{E_3}
 &&| E_1(\bm{\Psi},\bm{\Phi};\bm{\xi}_{hi}^{\sigma},\xi^{tr}_{hb})|
 \le|\langle \bm{\mathcal{I}}_{k+1}\bm{\Phi}-\bm{\Phi},\bm{\xi}_{hi}^{\sigma}\bm{n}\rangle_{\partial \mathcal{T}_h}
                                 |\nonumber\\
&&\quad\quad\quad+|\langle\beta(tr(\bm{Q}_k^i\bm{\Psi})-\mathcal{I}_1tr(\bm{\Psi})), tr(\bm{\xi}^{\sigma}_{hi})-\xi^{tr}_{hb} \rangle_{\partial\mathcal{T}_h/\partial\Omega}|\nonumber\\
&&\quad\quad\lesssim  h^{k+2}(|\bm{\sigma}|_{k+1}+\mu|\bm{u}|_{k+2})(\mu^{-1}|\bm{\Psi}|_1+|\bm{\Phi}|_2)
\end{eqnarray}
and
\begin{eqnarray}\label{E_4}
 &&|E_2(\bm{\Psi},\bm{\Phi};\bm{\xi}^u_{h}))|
 \le|\langle\alpha(\bm{Q}^i_{k+1}\bm{\Phi}-\bm{\mathcal{I}}_{k+1}\bm{\Phi}),
 \bm{\xi}^{u}_{hi}-\bm{\xi}^{u}_{hb}\rangle_{\partial\mathcal{T}_h}
                                 |\nonumber\\
&&\quad\quad\quad+|\langle\bm{\Psi}\bm{n}- \bm{\Psi}^i_k\bm{\sigma}\bm{n},\bm{\xi}^{u}_{hi}-\bm{\xi}^{u}_{hb}\rangle_{\partial\mathcal{T}_h}|\nonumber\\
&&\quad\quad\lesssim  h^{k+2}(|\bm{\sigma}|_{k+1}+\mu|\bm{\Phi}|_{k+2})|\bm{\Phi}|_2,
\end{eqnarray}
i.e.   \eqref{E1-est2} and \eqref{E3E4} hold.

\end{proof}
%

We are now ready to show the $L^2$-error estimation for the displacement approximation $\bm{u}_{hi}$.

\begin{theorem}  \label{th52} Let  $(\bm{\sigma},\bm{u})\in(\bm\Sigma\bigcap [H^{k+1}(\Omega)]^{d\times d})\times (\bm U\bigcap [H^{k+2}(\Omega)]^{d})$  be the solution to the model $(\ref{o1})$, and let
$\left(\bm{\sigma}_{hi}, \sigma_{hb}^{tr}, \bm{u}_h:=\{\bm{u}_{hi},\bm{u}_{hb}\}\right)\in  \bm{\Sigma}_{hi}\times \Sigma_{hb}^{tr}\times \bm{U}_{h}^{\bm{g}_D}$ and $\left(\bm{\sigma}_{hi},  \bm{u}_h\right)\in  \bm{\Sigma}_{hi}\times \bm{U}_{h}^{\bm{g}_D}$ be the solutions to the
 WG schemes \eqref{wg1new}-\eqref{wg2new} and \eqref{wg1}-\eqref{wg2}, respectively. Then,
under  the regularity assumption (\ref{rg}), it holds
\begin{eqnarray} \label{L2-est}
\|\bm{u}-\bm{u}_{hi}\|_0\lesssim h^{k+2}(\mu^{-1}|\bm{\sigma}|_{k+1}+|\bm{u}|_{k+2} ).
 \end{eqnarray}
\end{theorem}
\begin{proof} We only prove \eqref{L2-est} for   $k+1<d$, since the case of $k+1\ge d$ follows similarly.  Similar to the proof of  Lemma \ref{lemma42},   from \eqref{o2} we can easily obtain, for all $\bm{\tau}_{hi}\in \bm{\Sigma}_{hi} $,  $\tau^{tr}_{hb}\in\Sigma_{hb}^{tr}$, $ \bm{v}_h=\{\bm{v}_{hi},\bm{v}_{hb}\}\in  \bm{U}_{h}^{\bm{0}}$,
\begin{eqnarray}
a_h(\bm{Q}^i_k\bm{\Psi},\bm{\tau}_{hi})
+{z}_h(\bm{Q}^i_k\bm{\Psi},\mathcal{I}_1tr(\bm{\Psi});\bm{\tau}_{hi},\tau^{tr}_{hb})
-b_h(\bm{\tau}_{hi},\{\bm{Q}^i_{k+1}\bm{\Phi},\bm{\mathcal{I}}_{k+1}\bm{\Phi}\})\nonumber\\
=E_1(\bm{\Psi},\bm{\Phi};\bm{\tau}_{hi},\tau^{tr}_{hb}),\label{e21}
 \end{eqnarray}
\begin{eqnarray}\label{e22}
-b_h(\bm{Q}^i_k\bm{\Psi},\bm{v}_{h})-s_h(\{\bm{Q}^i_{k+1}\bm{\Phi},\bm{\mathcal{I}}_{k+1}\bm{\Phi}\},\bm{v}_h)=(\bm{\xi}^u_{hi},\bm{v}_{hi})+E_2(\bm{\Psi},\bm{\Phi};\bm{v}_{h}).
 \end{eqnarray}

Taking $\bm{v}_h=\bm{\xi}^u_h$ in (\ref{e22}), by \eqref{e11-21}, (\ref{e21})  and \eqref{e12} we have
\begin{eqnarray*}
\|\bm{\xi}^u_{hi}\|_0^2&=& -b_h(\bm{Q}^i_k\bm{\Psi},\bm{\xi}^u_{h})-s_h(\{\bm{Q}^i_{k+1}\bm{\Phi},\bm{\mathcal{I}}_{k+1}\bm{\Phi}\},\bm{\xi}^u_{h})
-E_2(\bm{\Psi},\bm{\Phi};\bm{\xi}^u_{h})\\
&=& -a_h(\bm{\xi}^{\sigma}_{hi},\bm{Q}^i_k\bm{\Psi})
-s_h(\{\bm{Q}^i_{k+1}\bm{\Phi},\bm{\mathcal{I}}_{k+1}\bm{\Phi}\},\bm{\xi}^u_{h})\nonumber\\
&&-{z}_h(\bm{\xi}^{\sigma}_{hi},\xi^{tr}_{hb};\bm{Q}^i_k\bm{\Psi},\mathcal{I}_1tr(\bm{\Psi}))\nonumber\\
&&+E_1(\bm{\sigma},\bm{u};\bm{Q}_k^i\bm{\Psi},\mathcal{I}_1tr(\bm{\Psi}))
-E_2(\bm{\Psi},\bm{\Phi};\bm{\xi}^u_{h})\\
&=&-b_h(\bm{\xi}^{\sigma}_{hi},\{\bm{Q}^i_{k+1}\bm{\Phi},\bm{\mathcal{I}}_{k+1}\bm{\Phi}\})\nonumber\\
&&-s_h(\{\bm{Q}^i_{k+1}\bm{\Phi},\bm{\mathcal{I}}_{k+1}\bm{\Phi}\},\bm{\xi}^u_{h})\nonumber\\
&&-E_1(\bm{\Psi},\bm{\Phi};\bm{\xi}_{hi}^{\sigma},\xi^{tr}_{hb}))  \nonumber\\
&&+E_1(\bm{\sigma},\bm{u};\bm{Q}_k^i\bm{\Psi},\mathcal{I}_1tr(\bm{\Psi}))
-E_2(\bm{\Psi},\bm{\Phi};\bm{\xi}^u_{h})\\
&=&E_2(\bm{\sigma},\bm{u};\{\bm{Q}^i_{k+1}\bm{\Phi},\bm{\mathcal{I}}_{k+1}\bm{\Phi}\})
-E_1(\bm{\Psi},\bm{\Phi};\bm{\xi}_{hi}^{\sigma},\xi^{tr}_{hb})) \nonumber\\
&&+E_1(\bm{\sigma},\bm{u};\bm{Q}_k^i\bm{\Psi},\mathcal{I}_1tr(\bm{\Psi}))
-E_2(\bm{\Psi},\bm{\Phi};\bm{\xi}^u_{h}),
\end{eqnarray*}
which, together with  Lemma \ref{lemma51} and the regularity (\ref{rg}), yields
 \begin{eqnarray*}
\|\bm{\xi}^u_{hi}\|_0^2&\lesssim&h^{k+2}(\mu^{-1}|\bm{\sigma}|_{k+1}+|\bm{u}|_{k+2})(|\bm{\Psi}|_1+\mu|\bm{\Phi}|_2) \nonumber\\
&\lesssim&h^{k+2}(\mu^{-1}|\bm{\sigma}|_{k+1}+|\bm{u}|_{k+2})\|\bm{\xi}^u_{hi}\|_0.
\end{eqnarray*}
As a result, the desired estimate follows from  the triangle inequality and  the approximation property of $\bm{Q}^i_{k+1}$.
\end{proof}

\section{Numerical examples}

In this section, we provide several numerical examples   to verify  our theoretical results. All tests  are programmed in C++ using the Eigen \cite{Eigen} library.\\

\subsection{A 2D example}

Let  $\Omega=(0,1)\times(0,1)$. We consider  the homogeneous Dirichlet boundary condition, and the exact solution $(\bm u, \bm\sigma)$ is of the following form:
\begin{eqnarray*}
\bm{u}=
\left(
   \begin{aligned}
\sin 2\pi y(-1+\cos 2\pi x)+\frac{1}{1+\lambda}\sin\pi x\sin\pi y\\
\sin 2\pi x(1-\cos 2\pi y)+\frac{1}{1+\lambda}\sin\pi x\sin\pi y
  \end{aligned}
\right),\\
\bm{\sigma}=
\left(
   \begin{aligned}
 2\mu\frac{d{u}_1}{dx}+\lambda (\frac{d{u}_1}{dx}+\frac{d{u}_2}{dy}),& &\mu(\frac{d{u}_1}{dy}+\frac{d{u}_2}{dx})\\
 \mu(\frac{d{u}_1}{dy}+\frac{d{u}_2}{dx}),& &2\mu\frac{du_2}{dy}+\lambda (\frac{d{u}_1}{dx}+\frac{d{u}_2}{dy}
  \end{aligned}
\right),
 \end{eqnarray*}
where  $\mu=1$ and $ \lambda=1,10^3,10^6$. Two types of meshes  are used (cf. Figures \ref{figure6}-\ref{figure7}). 
\begin{figure}[H]
\centering
\includegraphics[height=5.97cm ,width=12cm,angle=0]{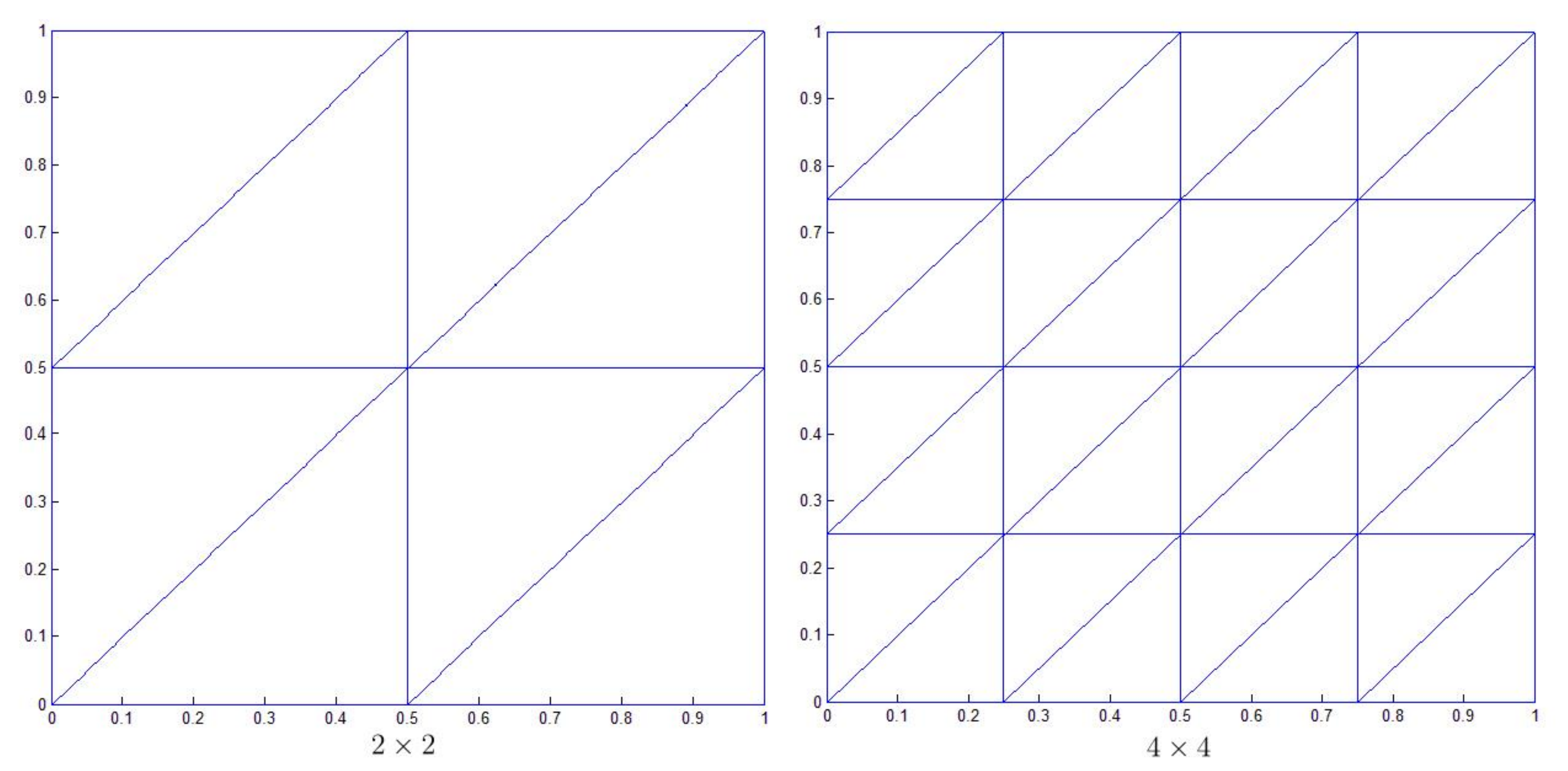}
\caption{\label{figure6} Triangle meshes}
\end{figure}

\begin{figure}[H]
\centering
\includegraphics[height=5.97cm ,width=12cm,angle=0]{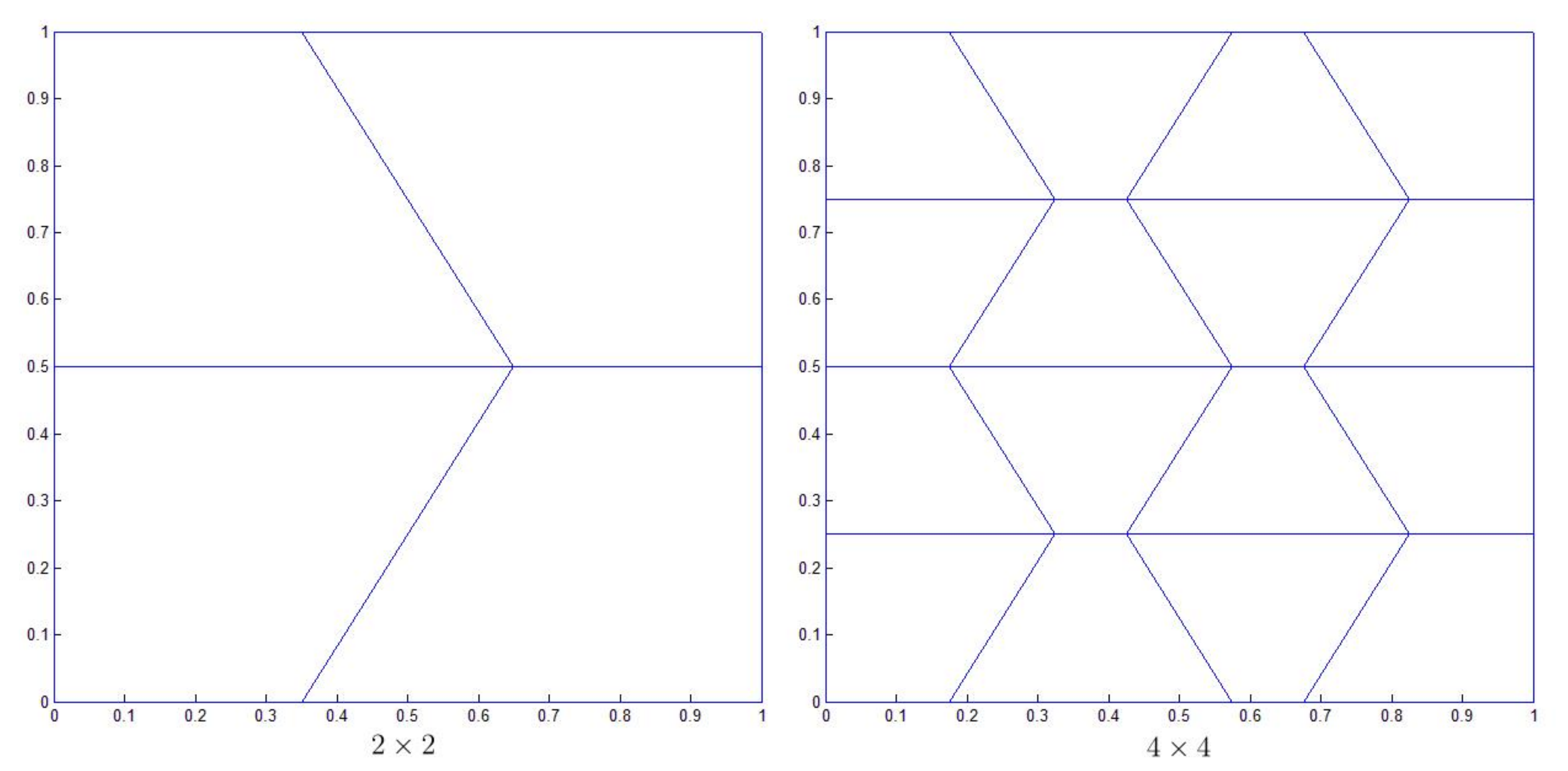}
\caption{\label{figure7} Ladder-shaped meshes}
\end{figure}

Numerical results  of the displacement and stress approximations are listed in Tables \ref{tab1}-\ref{tab2} for the proposed new WG methods \eqref{wg1new}-\eqref{wg2new} and \eqref{wg1}-\eqref{wg2} with $k=0,1,2$.  We can see that the   methods   yield optimal convergence rates that are uniformly with respect to the Lam\'e constant $\lambda$,  as   is conformable to the theoretical results.

For comparison we also list in Table \ref{tab1}  some results computed by the WG scheme \eqref{wg1*}-\eqref{wg2*}, denoted by WG*, with $k=0$  (cf. Remark \ref{rem2.2}). We note that for a fixed $k$, the new method is of  fewer degrees of freedom (DOF) than the corresponding WG* (after the local elimination).  We refer to Table \ref{tab1}(c) for the numbers of DOF for the two methods with $k=0$.


\begin{table}[H]
\small
\caption{Numerical results for $k=0$ on triangle meshes for a 2D example} \label{tab1}
\centering
\subtable[Displacement error $ {\|\bm{u}-\bm{u}_{hi}\|_0}/{\|\bm{u}\|_0}$]{
          \begin{tabular}{c|c|c|c|c|c|c|c}
\Xhline{1pt}

%

%
\multirow{2}{*}{Method}& 	&\multicolumn{2}{c|}{$\lambda=10^0$}	&\multicolumn{2}{c|}{$\lambda=10^3$}	
                &\multicolumn{2}{c}{$\lambda=10^6$}	\\
                \cline{2-8}
 &Mesh &Error &Rate
                &Error &Rate
                &Error &Rate
                \\
                     \cline{1-8}
\multirow{8}{*}{ {new WG}}&$2^2\times2^2$&2.5118E-01	& 	    &2.9130E-01	&	    &2.9137E-01	&\\
&$2^3\times2^3$&1.1458E-01	&1.13 	&1.5561E-01	&0.90 	&1.5572E-01	&0.90 \\
&$2^4\times2^4$&4.1279E-02	&1.47 	&6.1521E-02	&1.34 	&6.1579E-02	&1.34 	\\
&$2^5\times2^5$&1.2346E-02	&1.74 	&1.9154E-02	&1.68 	&1.9172E-02	&1.68 	\\
&$2^6\times2^6$&3.3351E-03	&1.89 	&5.2839E-03	&1.86 	&5.2892E-03	&1.86 	\\
&$2^7\times2^7$&8.6194E-04	&1.95 	&1.3782E-03	&1.94 	&1.3796E-03	&1.94 	\\
&$2^8\times2^8$&2.1870E-04	&1.98 	&3.5089E-04	&1.97 	&3.5126E-04	&1.97 	\\
&$2^9\times2^9$&5.5053E-05	&1.99 	&8.8436E-05	&1.99 	&8.8529E-05	&1.99 	\\

   \cline{1-8}
\multirow{8}{*}{  WG*}
&$2^2\times2^2$	&4.6433E-01	&	    &4.4558E-01	&	    &4.4551E-01 &\\
&$2^3\times2^3$	&1.3023E-01	&1.83 	&1.3145E-01	&1.76 	&1.3145E-01 &1.76\\
&$2^4\times2^4$	&3.4483E-02	&1.92 	&3.6053E-02	&1.87 	&3.6059E-02 &1.87\\
&$2^5\times2^5$	&8.7824E-03	&1.97 	&9.3198E-03	&1.95 	&9.3220E-03 &1.95\\
&$2^6\times2^6$	&2.2064E-03	&1.99 	&2.3522E-03	&1.99 	&2.3528E-03 &1.99\\
&$2^7\times2^7$	&5.5226E-04	&2.00 	&5.8949E-04	&2.00 	&5.8961E-04 &2.00\\
&$2^8\times2^8$	&1.3811E-04	&2.00 	&1.4746E-04	&2.00 	&1.4736E-04 &2.00\\

\Xhline{1pt}

\end{tabular}
}
\subtable[Stress error $ {\|\bm{\sigma}-\bm{\sigma}_{hi}\|_0}/{\|\bm{\sigma}\|_0}$]{
           \begin{tabular}{c|c|c|c|c|c|c|c}
\Xhline{1pt}

%

%
\multirow{2}{*}{Method}&	&\multicolumn{2}{c|}{$\lambda=10^0$}	&\multicolumn{2}{c|}{$\lambda=10^3$}	
                &\multicolumn{2}{c}{$\lambda=10^6$}	\\
                \cline{2-8}
 &Mesh &Error &Rate
                &Error &Rate
                &Error &Rate
                \\
                     \cline{1-8}
\multirow{8}{*}{ new WG}
&$2^2\times2^2$&5.8248E-01	&	    &5.9898E-01	&	    &5.9901E-01	&	    \\
&$2^3\times2^3$&3.2138E-01	&0.86 	&3.3902E-01	&0.82 	&3.3907E-01	&0.82 	\\
&$2^4\times2^4$&1.6203E-01	&0.99 	&1.7155E-01	&0.98 	&1.7157E-01	&0.98 	\\
&$2^5\times2^5$&7.9664E-02	&1.02 	&8.3487E-02	&1.04 	&8.3501E-02	&1.04 	\\
&$2^6\times2^6$&3.9321E-02	&1.02 	&4.0530E-02	&1.04 	&4.0535E-02	&1.04 	\\
&$2^7\times2^7$&1.9536E-02	&1.01 	&1.9902E-02	&1.03 	&1.9904E-02	&1.03 	\\
&$2^8\times2^8$&9.7405E-03	&1.00 	&9.8665E-03	&1.01 	&9.8668E-03	&1.01 	\\
&$2^9\times2^9$&4.8642E-03	&1.00 	&4.9152E-03	&1.01 	&4.9153E-03	&1.01 	\\

   \cline{1-8}
\multirow{8}{*}{WG*}

&$2^2\times2^2$	&5.4821E-01	&	    &5.9035E-01	&	    &5.9046E-01	&\\
&$2^3\times2^3$	&2.7265E-01	&1.01 	&2.9041E-01	&1.02 	&2.9046E-01	&1.02\\
&$2^4\times2^4$	&1.3062E-01	&1.06 	&1.3699E-01	&1.08 	&1.3700E-01	&1.08\\
&$2^5\times2^5$	&6.3998E-02	&1.03 	&6.6276E-02	&1.05 	&6.6280E-02	&1.05\\
&$2^6\times2^6$	&3.1804E-02	&1.01 	&3.2774E-02	&1.02 	&3.2776E-02	&1.02\\
&$2^7\times2^7$	&1.5877E-02	&1.00 	&1.6337E-02	&1.00 	&1.6338E-02	&1.00\\
&$2^8\times2^8$	&7.9351E-03	&1.00 	&8.1618E-03	&1.00 	&8.1622E-03	&1.00\\

\Xhline{1pt}

\end{tabular}
}

\subtable[Numbers of DOF  of different methods (after   local elimination)]{
           \begin{tabular}{c|c|c|c|c|c|c|c}
\Xhline{1pt}

%

%
Method
&$2^2\times 2^2$	
&$2^3\times 2^3$	
&$2^4\times 2^4$	
&$2^5\times 2^5$	
&$2^6\times 2^6$	
&$2^7\times 2^7$	
&$2^8\times 2^8$              \\

 \cline{1-8}

 {new WG}	    &75	&243	&867	&3267	&12675	&49923	&198147\\
 {WG*}	&224&832	&3200	&12544	&49664	&197632	&788480\\


\Xhline{1pt}

\end{tabular}
}

\end{table}

\begin{table}[H]
\small
\caption{Numerical results for new WG methods with $k=1$ on triangular meshes and $k=2$ on Ladder-shaped meshes: a 2D example} \label{tab2}
\centering
\subtable[Displacement error $ {\|\bm{u}-\bm{u}_{hi}\|_0}/{\|\bm{u}\|_0}$]{
          \begin{tabular}{c|c|c|c|c|c|c|c}
\Xhline{1pt}

%

%
\multirow{2}{*}{Method}& 	&\multicolumn{2}{c|}{$\lambda=10^0$}	&\multicolumn{2}{c|}{$\lambda=10^3$}	
                &\multicolumn{2}{c}{$\lambda=10^6$}	\\
                \cline{2-8}
 &Mesh &Error &Rate
                &Error &Rate
                &Error &Rate
                \\
                     \cline{1-8}
\multirow{6}{2cm}{  $k=1$: triangular meshes}
&$2^2\times2^2$	&5.2634E-02	&	    &6.2573E-02	&	    &6.2624E-02 &\\
&$2^3\times2^3$	&7.2373E-03	&2.86 	&8.7100E-03	&2.84 	&8.7173E-03 &2.84\\
&$2^4\times2^4$	&9.1526E-04	&2.98 	&1.1972E-03	&2.86 	&1.1987E-03 &2.86\\
&$2^5\times2^5$	&1.1056E-04	&3.05 	&1.6295E-04	&2.88 	&1.6328E-04 &2.88\\
&$2^6\times2^6$	&1.3369E-05	&3.05 	&2.1207E-05	&2.94 	&2.1263E-05 &2.94\\
&$2^7\times2^7$	&1.6481E-06	&3.02 	&2.6839E-06	&2.98 	&2.6930E-06 &2.98\\

                     \cline{1-8}
\multirow{5}{2cm}{  $k=2$: Ladder-shaped meshes}
&$2^2\times2^2$	&2.4966E-02	&	    &2.5643E-02	&	    &2.5643E-02	&\\
&$2^3\times2^3$	&1.7134E-03	&3.87 	&1.7813E-03	&3.85 	&1.7813E-03	&3.85\\
&$2^4\times2^4$	&1.1007E-04	&3.96 	&1.1508E-04	&3.95 	&1.1508E-04	&3.95\\
&$2^5\times2^5$	&6.9419E-06	&3.99 	&7.2747E-06	&3.98 	&7.2752E-06	&3.98\\
&$2^6\times2^6$	&4.3538E-07	&3.99 	&4.5666E-07	&3.99 	&4.6030E-07	&3.98\\


\Xhline{1pt}

\end{tabular}
}
\subtable[Stress error $ {\|\bm{\sigma}-\bm{\sigma}_{hi}\|_0}/{\|\bm{\sigma}\|_0}$]{
           \begin{tabular}{c|c|c|c|c|c|c|c}
\Xhline{1pt}

%

%
\multirow{2}{*}{Method}&	&\multicolumn{2}{c|}{$\lambda=10^0$}	&\multicolumn{2}{c|}{$\lambda=10^3$}	
                &\multicolumn{2}{c}{$\lambda=10^6$}	\\
                \cline{2-8}
 &Mesh &Error &Rate
                &Error &Rate
                &Error &Rate
                \\
                     \cline{1-8}
\multirow{6}{2cm}{  $k=1$: triangular meshes}
&$2^2\times2^2$	&1.5162E-01	&	    &2.1825E-01	&	    &2.1863E-01	&\\
&$2^3\times2^3$	&4.1901E-02	&1.86 	&6.4509E-02	&1.76 	&6.4643E-02	&1.76\\
&$2^4\times2^4$	&1.2563E-02	&1.74 	&2.2094E-02	&1.55 	&2.2151E-02	&1.55\\
&$2^5\times2^5$	&3.7221E-03	&1.75 	&7.6549E-03	&1.53 	&7.6789E-03	&1.53\\
&$2^6\times2^6$	&1.0091E-03	&1.88 	&2.3014E-03	&1.73 	&2.3098E-03	&1.73\\
&$2^7\times2^7$	&2.5899E-04	&1.96 	&6.1511E-04	&1.90 	&6.1749E-04	&1.90\\

                     \cline{1-8}
\multirow{5}{2cm}{  $k=2$: Ladder-shaped meshes}
&$2^2\times2^2$	&4.5366E-02	&	    &4.7317E-02	&	    &4.7322E-02	&\\
&$2^3\times2^3$	&5.8426E-03	&2.96 	&6.1579E-03	&2.94 	&6.1587E-03	&2.94\\
&$2^4\times2^4$	&7.3045E-04	&3.00 	&7.7307E-04	&2.99 	&7.7318E-04	&2.99\\
&$2^5\times2^5$	&9.1015E-05	&3.00 	&9.6541E-05	&3.00 	&9.6555E-05	&3.00\\
&$2^6\times2^6$	&1.1348E-05	&3.00 	&1.2046E-05	&3.00 	&1.2062E-05	&3.00\\

\Xhline{1pt}

\end{tabular}
}

\end{table}

\subsection{A 3D example}

Let $\Omega=(0,1)\times(0,1)\times (0,1)$ be subdivided into  simplicial meshes (cf. Figure \ref{figure8}).  We consider the homogeneous Dirichlet
boundary condition, and the exact solution $(\bm u, \bm\sigma)$ is of the following form:
\begin{align*}
u_1=&200(x-x^2)^2(2y^3-3y^2+y)(2z^3+3z^2+z),\\
u_2=&-100(y-y^2)^2(2x^3-3x^2+x)(2z^3-3z^2+z),\\
u_3=&-100(z-z^2)^2(2y^3-3y^2+y)(2x^3-3x^2+x),\\
\sigma_{11}=&400\mu(2x^3-3x^2+x)^2(2y^3-3y^2+y)(2z^3+3z^2+z),\\
\sigma_{22}=&-200\mu(2x^3-3x^2+x)^2(2y^3-3y^2+y)(2z^3+3z^2+z),\\
\sigma_{33}=&-200\mu(2x^3-3x^2+x)^2(2y^3-3y^2+y)(2z^3+3z^2+z),\\
\sigma_{12}=&\sigma_{21}=2\mu( \frac{d u_1}{d y}+\frac{d u_2}{dx} ),\\
\sigma_{13}=&\sigma_{31}=2\mu( \frac{d u_1}{d z}+\frac{d u_3}{dx} ),\\
\sigma_{23}=&\sigma_{32}=2\mu( \frac{d u_2}{d z}+\frac{d u_3}{dy} ),
 \end{align*}
where $\mu=0.5$ and $\lambda=10^0,10^3,10^6$.

 Numerical results  of the new WG methods with $k=0,1$ are  listed in Table \ref{tab3}. We can see that the methods yield uniformly optimal convergence rates, as      is conformable to the theoretical results.
\begin{figure}[H]
\centering
\includegraphics[height=5.97cm ,width=12cm,angle=0]{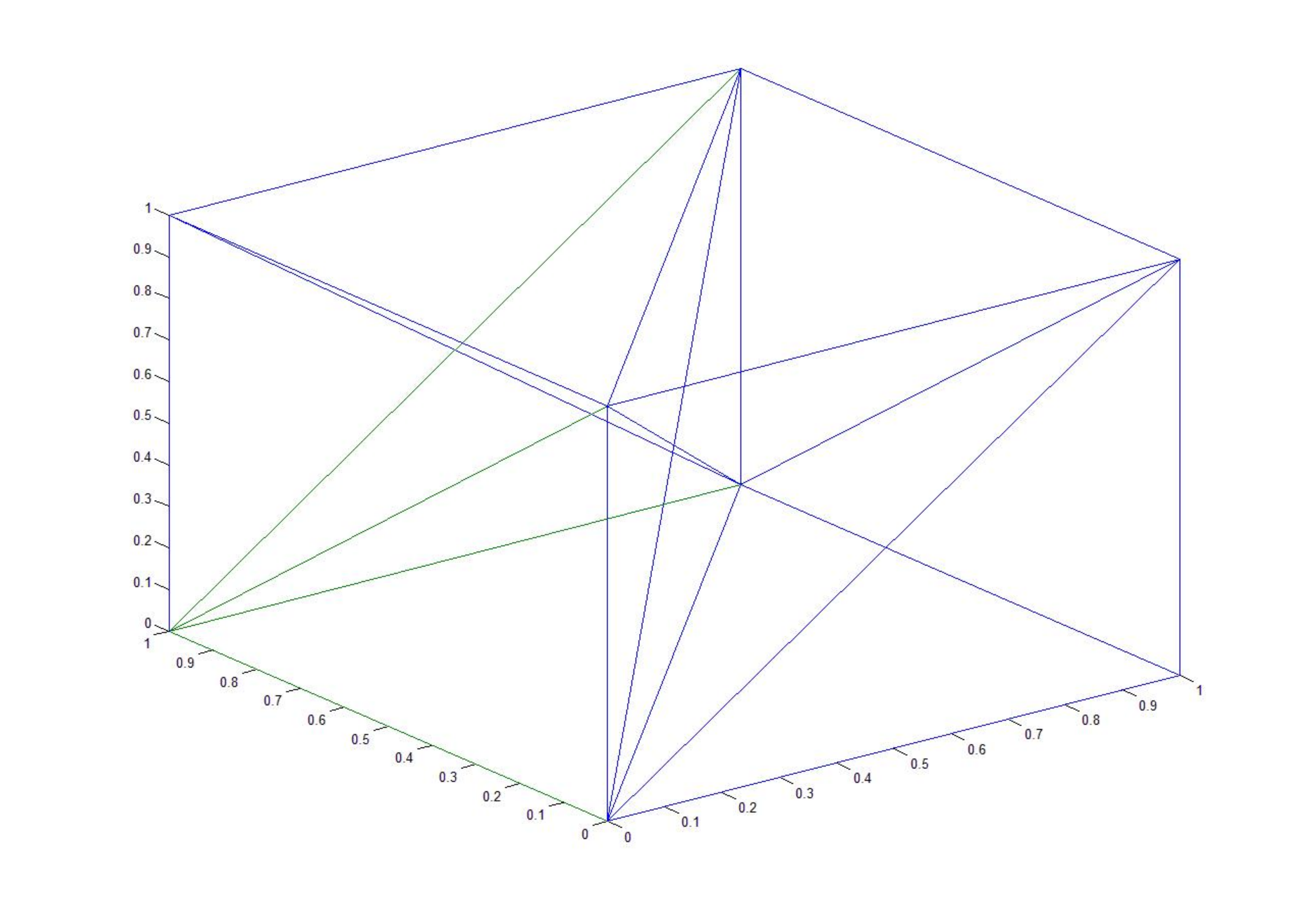}
\caption{\label{figure8} 3D simplicial mesh: $1\times1\times1$ (six elements)}
\end{figure}

\begin{table}[H]
\small
\caption{Numerical results for $k=0$ on simplicial meshes for a 3D example} \label{tab1-3}
\centering
 \begin{tabular}{c|c|c|c|c|c|c|c}
\Xhline{1pt}

%

%
\multirow{2}{*}{Method}& 	&\multicolumn{2}{c|}{$\lambda=10^0$}	&\multicolumn{2}{c|}{$\lambda=10^3$}	
                &\multicolumn{2}{c}{$\lambda=10^6$}	\\
                \cline{2-8}
 &Mesh &Error &Rate
                &Error &Rate
                &Error &Rate
                \\
                     \cline{1-8}
\multirow{8}{*}{\bf{$\frac{\|\bm{u}-\bm{u}_{hi}\|_0}{\|\bm{u}\|_0}$}}
&$4\times4\times4$	    &5.2896E-01	&	    &5.2877E-01	&	    &5.2877E-01	&\\
&$8\times8\times8$	    &2.3132E-01	&1.19 	&2.3085E-01	&1.20 	&2.3085E-01	&1.20\\
&$16\times16\times16$	&7.0628E-02	&1.71 	&7.0193E-02	&1.72 	&7.0192E-02	&1.72\\
&$32\times32\times32$	&1.8673E-02	&1.92 	&1.8463E-02	&1.93 	&1.8463E-02	&1.93\\
&$36\times36\times36$	&1.4813E-02	&1.97 	&1.4640E-02	&1.97 	&1.4640E-02	&1.97\\
&$40\times40\times40$	&1.2033E-02	&1.97 	&1.1891E-02	&1.97 	&1.1890E-02	&1.97\\
&$44\times44\times44$	&9.9662E-03	&1.98 	&9.8478E-03	&1.98 	&9.8476E-03	&1.98\\
&$48\times48\times48$	&8.3882E-03	&1.98 	&8.2891E-03	&1.98 	&8.2889E-03	&1.98\\

		                     \cline{1-8}
\multirow{8}{*}{$\frac{\|\bm{\sigma}-\bm{\sigma}_{hi}\|_0}{\|\bm{\sigma}\|_0}$}
&$4\times4\times4$	    &8.1612E-01	&	    &8.1625E-01	&	    &8.1625E-01	&\\
&$8\times8\times8$	    &5.1316E-01	&0.67 	&5.1347E-01	&0.67 	&5.1347E-01	&0.67\\
&$16\times16\times16$	&2.7874E-01	&0.88 	&2.7906E-01	&0.88 	&2.7906E-01	&0.88\\
&$32\times32\times32$	&1.4275E-01	&0.97 	&1.4296E-01	&0.97 	&1.4296E-01	&0.97\\
&$36\times36\times36$	&1.2710E-01	&0.99 	&1.2728E-01	&0.99 	&1.2728E-01	&0.99\\
&$40\times40\times40$	&1.1453E-01	&0.99 	&1.1468E-01	&0.99 	&1.1469E-01	&0.99\\
&$44\times44\times44$	&1.0421E-01	&0.99 	&1.0434E-01	&0.99 	&1.0434E-01	&0.99\\
&$48\times48\times48$	&9.5592E-02	&0.99 	&9.5705E-02	&0.99 	&9.5706E-02	&0.99\\


\Xhline{1pt}

\end{tabular}

\end{table}

\begin{table}[H]
\small
\caption{Numerical results for $k=1$ on simplicial meshes for a 3D example} \label{tab3}
\centering
          \begin{tabular}{c|c|c|c|c|c|c|c}
\Xhline{1pt}

%

%
\multirow{2}{*}{Method}& 	&\multicolumn{2}{c|}{$\lambda=10^0$}	&\multicolumn{2}{c|}{$\lambda=10^3$}	
                &\multicolumn{2}{c}{$\lambda=10^6$}	\\
                \cline{2-8}
 &Mesh &Error &Rate
                &Error &Rate
                &Error &Rate

                \\
                     \cline{1-8}
\multirow{5}{*}{$\frac{\|\bm{u}-\bm{u}_{hi}\|_0}{\|\bm{u}\|_0}$}
&$4\times4\times4$	    &8.6049E-02	&	    &8.7240E-02	&	    &8.7245E-02	&\\
&$8\times8\times8$	    &1.1340E-02	&2.92 	&1.1532E-02	&2.92 	&1.1533E-02	&2.92\\
&$12\times12\times12$	&3.2599E-03	&3.07 	&3.2960E-03	&3.09 	&3.2962E-03	&3.09\\
&$16\times16\times16$	&1.3332E-03	&3.11 	&1.3432E-03	&3.12 	&1.3432E-03	&3.12\\
&$20\times20\times20$	&6.6479E-04	&3.12 	&6.6836E-04	&3.13 	&6.6837E-04	&3.13\\

                     \cline{1-8}
\multirow{5}{*}{$\frac{\|\bm{\sigma}-\bm{\sigma}_{hi}\|_0}{\|\bm{\sigma}\|_0}$}
&$4\times4\times4$	    &2.7648E-01	&	    &2.8978E-01	&	    &2.8984E-01	&\\
&$8\times8\times8$	    &8.0878E-02	&1.77 	&8.2214E-02	&1.82 	&8.2220E-02	&1.82\\
&$12\times12\times12$	&3.8123E-02	&1.85 	&3.8374E-02	&1.88 	&3.8375E-02	&1.88\\
&$16\times16\times16$	&2.2266E-02	&1.87 	&2.2339E-02	&1.88 	&2.2340E-02	&1.88\\
&$20\times20\times20$	&1.4620E-02	&1.89 	&1.4648E-02	&1.89 	&1.4649E-02	&1.89\\


\Xhline{1pt}

\end{tabular}

\end{table}

\addcontentsline{toc}{section}{References.}
\label{}
\bibliographystyle{siam}






\appendix
\appendixpage
\addcontentsline{toc}{section}{Appendices}\markboth{APPENDICES}{}
\begin{appendices}

\section{\label{A1} Modified Scott-Zhang interpolation for vectors}

Let $\Omega\color{black}\subset\color{black}\mathbb{R}^d \ (d=2,3)$ be a polyhedral region with boundary $\partial\Omega={\Gamma_D\cup\Gamma_N}$, where $meas(\Gamma_D)>0$ and $\Gamma_D\cap\Gamma_N=\emptyset$.

Let $\mathcal{T}_h$ be a shape regular simplicial subdivision of $\Omega$ with maximum mesh size $h$,
such that each (open) boundary edge (face) belongs either to $\Gamma_D$, or to $\Gamma_N$. Furthermore, we assume that there should be at least 2 edges (or faces) on $\Gamma_S$ if $\Gamma_S\neq\emptyset$, where  $S=D,N$.

We shall construct an interpolation operator
$${\mathcal{I}}_{k+1}: W^{l,p}(\Omega)\to W^{l,p}(\Omega) \cap \mathbb{P}_{k+1}(\mathcal{T}_h),$$
where
\begin{eqnarray}
l\ge 1 \text{ if } p=1 \text{ and } l>1/p \text{ otherwise}.
\end{eqnarray}
 To this end, let $\mathcal{N}_h=\{A_i\}^N_{i=1}$ be the set of all
interpolation nodes of $\mathcal{T}_h$ and $\{\bm{\phi}_i\}_{i=1}^N$ be the corresponding nodal basis of
$\mathbb{P}_{k+1}(\mathcal{T}_h)\cap H^1(\Omega)$. And let $\{A_{J_i}\}_{i=0}^{M+1}\subset \mathcal{N}_h$ be the set of vertexes of edge (or face) $E\subset \Gamma_N$, and $\{A_{M_i}\}_{i=1}^{S}\subset \mathcal{N}_h$ be the set of interior nodes of edge (or face) $E\subset \Gamma_N$. If $d=2$ the points $A_{J_0}$ and $A_{J_{M+1}}$ should be the adjoint point of $\Gamma_D$ and $\Gamma_N$, and $A_{J_i}$ $(i=1,2,\cdots,M)$ is between $A_{J_{i-1}}$ and $A_{J_{i+1}}$.

The interpolation operator
${\mathcal{I}}_{k+1}$ is defined as follows:
For any $v\in W^{l,p}(\Omega)$,
given by
\begin{eqnarray}\label{def-interpolation}
 {\mathcal{I}}_{k+1} {v}:=\sum_{A_i\in\mathcal{N}_h} {\mathcal{I}}_{k+1} {v}( {A}_i) {\phi}_i.
\end{eqnarray}
Here for any node $A_i$, the value ${\mathcal{I}}_{k+1} {v}( {A}_i)$ is determined by the following way, i.e.  {\bf (1)-(4)}.

  {\bf(1)} If $A_{i}$ is an interior point of some $d$-simplex $T\in \mathcal{T}_h$, then let $\{A_{i,j}\}_{j=1}^{n_0}$ ($A_{i,1}=A_{i}$, $n_0=C^{k+1}_{k+d+1}$) be the set of nodal points in $T$ and $\{ {\phi}_{i,j}\}_{j=1}^{n_0}$ be the corresponding nodal basis, and let
$\{ {\psi}_{i,j}\}_{j=1}^{n_0}$ be the $L^2(T)$-dual basis  of $\{ {\phi}_{i,j}\}_{j=1}^{n_0}$ satisfying
\begin{eqnarray}
( {\phi}_{i,j}, {\psi}_{i,k})_T=\delta_{jk},
\end{eqnarray}
where $\delta_{jk}$ is the Kronecker delta.  In this case,  we define 
\begin{eqnarray}
 {\mathcal{I}}_{k+1} {v}(A_i):=( {v}, {\psi}_{i,1})_T.
\end{eqnarray}

  {\bf(2)} If $A_{i}$ is an interior point  of edge (or face) ($(d-1)-$simplex) $E\subset \partial T$, then
let $\{A_{i,j}\}_{j=1}^{n_1}$ ($A_{i,1}=A_{i}$, $n_1=C^{k+1}_{k+d}$) be the set of nodal points in $E$ and $\{ {\phi}_{i,j}\}_{j=1}^{n_1}$ be the corresponding nodal basis, and let
$\{ {\psi}_{i,j}\}_{j=1}^{n_1}$ be the $L^2(E)$-dual basis  of $\{ {\phi}_{i,j}\}_{j=1}^{n_1}$ satisfying
\begin{eqnarray}
\langle {\phi}_{i,j}, {\psi}_{i,k}\rangle_E=\delta_{jk}.
\end{eqnarray}
Thus, we define 
\begin{eqnarray}
 {\mathcal{I}}_{k+1} {v}(A_i):=\langle {v}, {\psi}_{i,1}\rangle_E.
\end{eqnarray}

  {\bf (3)} For the rest of  $A_i\in\mathcal{N}_h$, we select a  $(d-1)$-simplex $E$ such that
$A_i\in \overline{E}$, subject only to the restriction
\begin{eqnarray}
E\subset\overline{\Gamma_D} \text{ if } A_i\in \overline{\Gamma_D}.
\end{eqnarray}
Let $\{A_{i,j}\}_{j=1}^{n_1}$ ($A_{i,1}=A_{i}$, $n_1=C^{k+1}_{k+d}$) be the set of nodal points in $\overline{E}$ and $\{ {\phi}_{i,j}\}_{j=1}^{n_1}$ be the corresponding nodal basis, and let
$\{ {\psi}_{i,j}\}_{j=1}^{n_1}$ be the $L^2(E)$-dual basis  of $\{ {\phi}_{i,j}\}_{j=1}^{n_1}$ satisfying
\begin{eqnarray}
\langle {\phi}_{i,j}, {\psi}_{i,k}\rangle_E=\delta_{jk}.
\end{eqnarray}
Then we define 
\begin{eqnarray}
 {\mathcal{I}}_{k+1} {v}(A_i):=\langle {v}, {\psi}_{i,1}\rangle_E.
\end{eqnarray}

  {\bf (4)}   We need to  modify the interpolation conditions of $ {\mathcal{I}}_{k+1} {v}$ at some nodes $A_i$ in $\partial \Omega$ for the  three cases  (4a)-(4c) below. We note that the  (4b)-(4c) are corresponding to the case of $k+1<d$.
\begin{itemize}
\item[(4a)] If $k+1\ge d$, then, for any edge (or face, $d-1$-simplex) $E\subset  \Gamma_S$ ($S=D$ or $N$), there always exists at least one interior node in $E$. We choose one such  node,  $A_i$, and replace the  corresponding interpolation condition of  $ {\mathcal{I}}_{k+1} {v}(A_i)$ in  {\bf (2)} or {\bf (3)}   by
\begin{eqnarray}
\langle {\mathcal{I}}_{k+1} {v}, 1 \rangle_{E}
=\langle {v}, 1 \rangle_{E}.
\end{eqnarray}

\item[ (4b)] If $k+1= d-1$, then, for any open $(d-2)$-simplex $P\in \Gamma_S$ ($S=D$ or $N$), there exists only one node $A_i$ inside $\overline{P}$. Let $\{A_{I_j}\}_{j=1}^M$ denote
  the set of   all nodes inside all the  $(d-2)$-simplexes  in $\Gamma_S$.   For any $A_i\in \{A_{I_j}\}_{j=1}^M$, there exist exactly two $(d-1)$-simplexes $E_{i,1}\subset \Gamma_S $ and $E_{i,2}\subset \Gamma_S$ that are adjoined at $A_i$. Set $w(A_i):=\overline{E_{i,1}}\bigcup \overline{E_{i,2}}$. We choose
a set of nodes $\{A_{J_i}\}_{i=1}^{l}\subset \{A_{I_j}\}_{j=1}^M$ such that
\begin{eqnarray}
\mu_{d-1}(w(A_{J_j})\bigcap w(A_{J_k}))=0,j\neq k,\\
\bigcup_{A_{J_i}\in \{A_{J_i}\}_{i=1}^{l}}w(A_{J_i})=\overline{\Gamma_S}.
\end{eqnarray}
Then the interpolation condition of $ {\mathcal{I}}_{k+1}v$ at each $A_i\in \{A_{J_i}\}_{i=1}^{l}$ is replaced by
\begin{eqnarray}
\langle {\mathcal{I}}_{k+1} {v}, 1 \rangle_{E_{i,1}}
+\langle {\mathcal{I}}_{k+1} {v}, 1 \rangle_{E_{i,2}}
=\langle {v}, 1 \rangle_{E_{i,1}}+\langle {v}, 1 \rangle_{E_{i,2}}.
\end{eqnarray}

\item[ (4c)] If $k+1=1$ and $d=3$, then let $\{A_{I_j}\}_{j=1}^M$   denote the set of all nodes inside $\Gamma_S$ ($S=D $ or $N$). For any $A_i\in \{A_{I_j}\}_{j=1}^M$, there exist   $l_i$ $(d-1)$-simplexes, $\{E_{i,j}\}_{j=1}^{l_i}$, in $\Gamma_S $ that
  adjoined at $A_i$.  Set $w(A_i):=\bigcup_{E_{i,j}\in \{E_{i,j}\}_{j=1}^{l_i}}
\overline{E_{i,j}}$. We choose
a set of nodes $\{A_{J_i}\}_{i=1}^{l}\subset \{A_{I_j}\}_{j=1}^M$ such that
\begin{eqnarray}
\mu_{d-1}(w(A_{J_j})\bigcap w(A_{J_k}))=0,j\neq k,\\
\bigcup_{A_{J_i}\in \{A_{J_i}\}_{i=1}^{l}}w(A_{J_i})=\overline{\Gamma_S}.
\end{eqnarray}
Then  the interpolation condition of $ {\mathcal{I}}_{k+1}v$ at each $A_i\in \{A_{J_i}\}_{i=1}^{l}$ is replaced by
\begin{eqnarray}
\sum_{E_{i,j}\in \{E_{i,j}\}_{j=1}^{l_i} }\langle {\mathcal{I}}_{k+1} {v}, 1\rangle_{E_{i,j}}
=\sum_{E_{i,j}\in \{E_{i,j}\}_{j=1}^{l_i} }\langle {v}, 1 \rangle_{E_{i,j}}.
\end{eqnarray}
\end{itemize}

Based on  interpolation conditions  in  {\bf(1)-(4)},  we know that the interpolation polynomial  $ {\mathcal{I}}_{k+1}v \in W^{l,p}(\Omega) \cap \mathbb{P}_{k+1}(\mathcal{T}_h)$ is uniquely determined.

In view of the definition of the operator $ {\mathcal{I}}_{k+1}$, i.e. \eqref{def-interpolation}, we easily have the following results.
\begin{lemma} For any $ {v}\in W^{l,p}(\Omega)$, there   holds
\begin{eqnarray}
\langle {\mathcal{I}}_{k+1} {v}, 1 \rangle_{\Gamma_S}=\langle {v}, 1 \rangle_{\Gamma_S},\quad S=D,N,\\
 {\mathcal{I}}_{k+1} {v}|_{\Gamma_D}= {0} \text{ if }  {v}|_{\Gamma_D}= {0}.
\end{eqnarray}
In particular,
\begin{eqnarray}\label{invariance}
 {\mathcal{I}}_{k+1} {v}= {v}, \quad \forall {v}\in \mathbb{P}_{k+1}(\mathcal{T}_h)\cap W^{l,p}(\Omega).
\end{eqnarray}

\end{lemma}

\begin{lemma}\label{lem-A.2}
 For any $ {v}\in W^{l,p}(\Omega)$ and  $T\in\mathcal{T}_h$, there holds
\begin{eqnarray}
\| {\mathcal{I}}_{k+1} {v}\|_{m,q,T}\lesssim \sum_{r=0}^lh_T^{r-m+d/q-d/p}| {v}|_{r,p,S(T)},
\end{eqnarray}
where
\begin{eqnarray}
S(T)=S^{'}(T):=\text{interior}(\bigcup\{\overline{T^{'}}|\overline{T^{'}}\bigcap
\overline{T}\neq\emptyset, T^{'}\in\mathcal{T}_h \})
\end{eqnarray}
if $k+1\ge d$, and
\begin{eqnarray}
S(T)=\text{interior}(\bigcup\{\overline{T^{'}}|\overline{T^{'}}\bigcap
\overline{S^{'}(T)}\neq\emptyset, T^{'}\in\mathcal{T}_h \})
\end{eqnarray}
if $ k+1< d$.
\end{lemma}
\begin{proof} For   $ {v}\in W^{l,p}(\Omega)$ and $A_i\in\mathcal{N}_h$, by (1) we have
\begin{eqnarray}
| {\mathcal{I}}_{k+1} {v}(A_i)|&\le&\| {v}\|_{0,1,T}\| {\psi}_{i,1}\|_{0,\infty,T}\nonumber\\
                                   &\le& (h_T^{-1})^0 (h_T^{d/1})|\hat{ {v}}|_{0,1,\hat{T}}h_T^{-d} \nonumber\\
                                   &\lesssim& \|\hat{ {v}}\|_{l,p,\hat{T}} \nonumber\\
                                   &\lesssim&\sum_{r=0}^lh_T^{r-d/p} | {v}|_{r,p,T}.
\end{eqnarray}
From (2)-(3) we get
\begin{eqnarray}
| {\mathcal{I}}_{k+1} {v}(A_i)|&\le&\| {v}\|_{0,1,E}\| {\psi}_{i,1}\|_{0,\infty,E}\nonumber\\
                                   &\le& (h_T^{-1/1}) (h_T^{d/1})|\hat{ {v}}|_{0,1,\hat{T}}h_T^{-(d-1)} \nonumber\\
                                   &\lesssim& \|\hat{ {v}}\|_{l,p,\hat{T}} \nonumber\\
                                   &\lesssim&\sum_{r=0}^lh_T^{r-d/p} | {v}|_{r,p,T}.
\end{eqnarray}
For the case (4a), $ {\mathcal{I}}_{k+1} {v}$ on $E$ has the form
\begin{eqnarray}
 {\mathcal{I}}_{k+1} {v}|_E= {\mathcal{I}}_{k+1} {v}(A_i) {\phi}_i
+\sum_{j=1}^{d+1} {\mathcal{I}}_{k+1} {v}(A_{i,j}) {\phi}_{i,j}.
\end{eqnarray}
Therefore
\begin{eqnarray}
\langle {\mathcal{I}}_{k+1} {v},1\rangle_E= {\mathcal{I}}_{k+1} {v}(A_i)\langle {\phi}_i,1\rangle_E
+\sum_{j=1}^{d+1} {\mathcal{I}}_{k+1} {v}(A_{i,j})\langle {\phi}_{i,j} ,1\rangle_E,
\end{eqnarray}
which leads to
\begin{eqnarray}
 {\mathcal{I}}_{k+1} {v}(A_i)=&&\frac{1}{\langle {\phi}_i,1\rangle_E}\langle {\mathcal{I}}_{k+1} {v},1\rangle_E
-\sum_{j=1}^{d+1} {\mathcal{I}}_{k+1} {v}(A_{i,j})\frac{\langle {\phi}_{i,j},1\rangle_E}{\langle {\phi}_i,1\rangle_E} \nonumber\\
&=&\frac{1}{\langle {\phi}_i,1\rangle_E}\langle  {v},1\rangle_E
-\sum_{j=1}^{d+1} {\mathcal{I}}_{k+1} {v}(A_{i,j})\frac{\langle {\phi}_{i,j},1\rangle_E}{\langle {\phi}_i,1\rangle_E}.
\end{eqnarray}
So
\begin{eqnarray}
 |{\mathcal{I}}_{k+1} {v}(A_i)|&\lesssim&\frac{1}{|E|}|\langle  {v},1\rangle_E|
-\sum_{j=1}^{d+1}| {\mathcal{I}}_{k+1} {v}(A_{i,j})|\nonumber\\
&\lesssim& h_E^{-(d-1)}|v|_{0,1,E}+\sum_{r=0}^lh_T^{r-d/p} | {v}|_{r,p,S(T)}\nonumber\\
&\lesssim &\sum_{r=0}^lh_T^{r-d/p} | {v}|_{r,p,S(T)}.
\end{eqnarray}
Similarly, for the cases (4b)-(4c), it holds
\begin{eqnarray}
 {\mathcal{I}}_{k+1} {v}(A_i)\lesssim\sum_{r=0}^lh_T^{r-d/p} | {v}|_{r,p,S(T)}.
\end{eqnarray}
As a result, from \eqref{def-interpolation} it follows
\begin{eqnarray}
\| {\mathcal{I}}_{k+1} {v}\|_{m,q,T}&\le& \sum_{i=1}^{n_1}
| {\mathcal{I}}_{k+1} {v}(A_i)|
\| {\phi}_i\|_{m,q,T}\nonumber\\
&\lesssim&\sum_{r=0}^lh_T^{r-m+d/q-d/p} | {v}|_{r,p,S(T)}.
\end{eqnarray}
\color{red}This completes the proof.\color{black}
\end{proof}

In light of \eqref{invariance}, Lemma \ref{lem-A.2}, and the triangle inequality, we easily obtain  the following approximation result.

\begin{lemma} For any $ {v}\in W^{l,p}(\Omega)$ and $0\le m\le k+2$, we have
\begin{eqnarray}
\| {v}- {\mathcal{I}}_{k+1} {v}\|_{m,p,T}
\lesssim\sum_{r=0}^mh_T^{r-m} \inf_{ {w}\in
 [\mathbb{P}_{k+1}(\mathcal{T}_h)\bigcap H^1(\Omega)]^d }\| {v}- {w}\|_{r,p,S(T)}.
\end{eqnarray}
\end{lemma}


\section{\label{B1} Inf-sup conditions on $[H^1_D(\Omega)]^d\times L^2(\Omega)$}

We first cite a result from \cite{Bramble;2001}:

\begin{lemma}  Assume that meas$(\Gamma_D)>0$. Then there exists a positive $C$ such that
\begin{eqnarray}
\sup_{\bm{v}\in  [H^1_D(\Omega)]^d} \frac{(\nabla\cdot\bm{v},q)}{||\bm{v}||_1}
\ge C\|q\|_0, \quad \forall q\in  L^2(\Omega).\label{B00}
\end{eqnarray}

\end{lemma}
\begin{lemma}Assume that meas$(\Gamma_D)>0$. Then  it holds the   inf-sup condition
\begin{eqnarray}
\sup_{\bm{v}_{h}\in [\mathbb{P}_{1}(\mathcal{T}_h)\cap H^1_D(\Omega)]^d}\frac{(\nabla\cdot\bm{v}_h,\overline{p})}{|\bm{v}_h|_1}
\gtrsim \|\overline{p}\|_0, \quad \text{ for any constant } \overline{p}.\label{2.39}
\end{eqnarray}
\end{lemma}
\begin{proof} From   \eqref{B00} and Lemma \ref{lemma23} it follows
\begin{eqnarray}
\|\overline{p}\|_0&\lesssim&\sup_{\bm{v}\in H^1_D(\Omega)]^d}\frac{(\nabla\cdot\bm{v},\overline{p})}{|\bm{v}|_1}\nonumber\\
&=&\sup_{\bm{v}\in [H^1_D(\Omega)]^d}\frac{(\nabla\cdot\bm{\mathcal{I}}_{1}\bm{v},\overline{p})}{|\bm{v}|_1}
+\sup_{\bm{v}\in [H^1_D(\Omega)]^d}\frac{(\nabla\cdot(\bm{v}-\bm{\mathcal{I}}_{1}\bm{v}),\overline{p})}{|\bm{v}|_1}\nonumber\\
&=&\sup_{\bm{v}\in [H^1_D(\Omega)]^d}\frac{(\nabla\cdot\bm{\mathcal{I}}_{1}\bm{v},\overline{p})}{|\bm{v}|_1}\nonumber\\
&\lesssim&\sup_{\bm{v}\in H^1_D(\Omega)]^d}\frac{(\nabla\cdot\bm{\mathcal{I}}_{1}\bm{v},\overline{p})}{|\bm{\mathcal{I}}_{1}\bm{v}|_1},\nonumber
\end{eqnarray}
which yields the desired  result.
\end{proof}

\begin{theorem} \label{thmB.3} Let  $\bm{V}_h\subset [H^1_D(\Omega)]^d$ and $Q_h\subset L^2(\Omega)$ be two finite dimensional spaces such that
$[\mathbb{P}_{1}(\mathcal{T}_h)\bigcap H^1_D(\Omega)]^d\subset \bm{V}_h$ and
\begin{eqnarray}
\sup_{\bm{v}_{h}\in\bm{V}_h\bigcap [H^1_0(\Omega)]^d}\frac{(\nabla\cdot\bm{v}_h,p_h)}{|\bm{v}_h|_1}
\gtrsim \|p_h\|_0, \forall p_h\in Q_h\bigcap L^2_0(\Omega). \label{2.41}
\end{eqnarray}
Then  the following inf-sup condition holds:
\begin{eqnarray}
\sup_{\bm{v}_{h}\in \bm{V}_h}\frac{(\nabla\cdot\bm{v}_h,p_h)}{|\bm{v}_h|_1}
\gtrsim\|p_h\|_0, \forall p_h\in Q_h.
\end{eqnarray}
\end{theorem}
\begin{proof} From \eqref{2.39} and \eqref{2.41} we know that, for all $p_{h}\in Q_h$, there exists $\bm{v}_1\in [\mathbb{P}_{1}(\mathcal{T}_h)\cap H^1_D(\Omega)]^d$, $\bm{v}_2\in\bm{V}_h\bigcap [H^1_0(\Omega)]^d$,  and a positive constant $C_0$, independent of $p_{h}$,  $\bm{v}_1$,  $\bm{v}_2$, and the mesh size $h$, such that
\begin{eqnarray}
&&(\nabla\cdot\bm{v}_1,\overline{p}_h)=\|\overline{p}_h\|^2_0,\quad |\bm{v}_1|_1\le C_0 \|\overline{p}_h\|_0,\\
&&(\nabla\cdot\bm{v}_2,p_h-\overline{p}_h)=\|p_h-\overline{p}_h\|^2_0,\quad |\bm{v}_2|_1\le C_0 \|p_h-\overline{p}_h\|_0,
\end{eqnarray}
where $\overline{p}_h=(p_h,1)$. Then we can take
$\bm{v}_h=(\bm{v}_1+\tilde\alpha\bm{v}_2 )\in \bm{V}_h$ with
\begin{eqnarray}
(\nabla\cdot\bm{v}_h,p_h)&=&(\nabla\cdot\bm{v}_1,p_h)+\tilde\alpha(\nabla\cdot\bm{v}_2,p_h) \nonumber\\
                          &=&(\nabla\cdot\bm{v}_1,\overline{p}_h)+(\nabla\cdot\bm{v}_1,p_h-\overline{p}_h)+\tilde\alpha(\nabla\cdot\bm{v}_2,p_h-\overline{p}_h)
\nonumber\\
&\ge&\|\overline{p}_h\|^2_0+\tilde\alpha\|p_h-\overline{p}_h\|^2_0-C_0\|\overline{p}_h\|_0\|p_h-\overline{p}_h\|_0 \nonumber\\
&\ge&\frac{1}{2}\|\overline{p}_h\|^2_0+(\tilde\alpha-\frac{C_0^2}{2})\|p_h-\overline{p}_h\|^2_0 \nonumber\\
&\ge&\frac{1}{4}\|p_h\|^2_0,\label{2.45}
\end{eqnarray}
where $\tilde\alpha=\frac{C_0^2}{2}+\frac{1}{2}$.
On the other hand,
\begin{eqnarray}
|\bm{v}_h|_1\le |\bm{v}_1|_1+\tilde\alpha|\bm{v}_2|_1\le C_0\left(  \|\overline{p}_h\|_0+\tilde\alpha\|p_h-\overline{p}_h\|_0\right)\lesssim \|p_h\|_0,
\end{eqnarray}
which, together with   \eqref{2.45}, implies the desired conclusion.
\end{proof}

\end{appendices}

\end{document}